\documentclass[oneside,reqno,11pt]{amsart}

\usepackage{graphicx}
\usepackage{epstopdf}
\usepackage{amsmath}
\usepackage{amssymb}
\usepackage[top=1in, bottom=1.25in, left=1.25in, right=1.25in]{geometry}
\usepackage{textcomp}
\usepackage{mathrsfs}
\usepackage{tikz}
\usepackage{tikz-cd}
\usepackage[utf8]{inputenc}
\usepackage{mathabx}
\usepackage{bm}
\usepackage{amsthm}
\usepackage{amsfonts}
\usepackage[colorlinks=true,linkcolor=blue]{hyperref}
\usepackage{pb-diagram}
\usepackage{lipsum}
\usepackage{secdot} 
\usepackage{color}
\usepackage[normalem]{ulem}
\usepackage{enumitem}
\usepackage{appendix} 
\usepackage{color}
\usepackage{calc}
\usepackage{slantsc}
\usepackage[sc]{mathpazo}
\usepackage{indentfirst}
\usepackage{mathtools}
\usepackage{hyperref}
\usepackage{amscd}
\usepackage{caption}
\usepackage{subcaption}
\usepackage{fancybox}
\usepackage{wrapfig}
\usepackage[all]{xy}
\usepackage{verbatim}
\usepackage{stmaryrd}

\theoremstyle{plain}
\newtheorem{theorem}{Theorem}[section]
\newtheorem{lemma}[theorem]{Lemma}
\newtheorem{corollary}[theorem]{Corollary}

\newtheorem{prop}[theorem]{Proposition}
\newtheorem{defn}[theorem]{Definition}
\newtheorem{example}[theorem]{Example}
\newtheorem{remark}[theorem]{Remark}

\numberwithin{equation}{section}

\newcommand{\Spec}{\operatorname{Spec}}
\newcommand{\Proj}{\operatorname{Proj}}

\newcommand{\Image}{\operatorname{Im}}

\newcommand{\twobar}{/\kern-0.5em/}
\newcommand{\threebar}{/\kern-0.5em/\kern-0.5em/}

\newcommand{\Z}{\mathbb{Z}}
\newcommand{\Q}{\mathbb{Q}}
\newcommand{\R}{\mathbb{R}}
\newcommand{\C}{\mathbb{C}}

\newcommand{\bP}{\mathbb{P}}

\newcommand{\into}{\hookrightarrow}

\newcommand{\Hom}{\mathrm{Hom}}

\newcommand{\cM}{\mathcal{M}}

\newcommand{\cO}{\mathcal{O}}
\newcommand{\cL}{\mathcal{L}}
\newcommand{\pt}{\mathrm{pt}}

\newcommand{\Pic}{\mathrm{Pic}}

\newcommand{\isoto}{\xrightarrow{\sim}}

\renewcommand{\Im}{\operatorname{Im}}

\newcommand{\Res}{\operatorname{Res}}

\newcommand{\bfM}{\textbf{M}}
\newcommand{\bfN}{\textbf{N}}
\newcommand{\cN}{\mathcal{N}}

\newcommand{\rK}{\mathrm{K}}
\newcommand{\rT}{\mathrm{T}}
\newcommand{\rA}{\mathrm{A}}

\makeatletter
\def\subsubsection{\@startsection{subsubsection}{3}%
	\z@{.5\linespacing\@plus.7\linespacing}{-.5em}%
	{\normalfont\bfseries}}
\makeatother

\title{Elliptic hypertoric varieties}

\author{Naichung Conan Leung}
\address{The Institute of Mathematical Sciences and Department of Mathematics, The Chinese University of Hong Kong, Shatin, Hong Kong}
\email{leung@math.cuhk.edu.hk}

\author{Xiao Zheng}
\address{The Institute of Mathematical Sciences, The Chinese university of Hong Kong, Shatin, Hong Kong}
\email{xiaozh259@gmail.com}

\begin{document}
	
\begin{abstract}
We introduce elliptic hypertoric varieties, which is an elliptic analogue of hypertoric varieties and multiplicative hypertoric varieties. 
We also prove an elliptic version of Hikita conjecture, which relates elliptic (resp. additive and mulitplicative) hypertoric varieties 
to the equivariant elliptic cohomology (resp. ordinary cohomology and K-theory) of their 3d mirror hypertoric varieties.
%We introduce in this paper an elliptic analogue of complex Hamiltonian reduction for abelian group action, which we use to define elliptic hypertoric varieties. We also prove an elliptic version of the Hikita conjecture for hypertoric varieties.
%We introduce the notion of ``elliptic''-Hamiltonian $\rT$-manifolds for which the moment maps are valued in elliptic curves, and their Hamiltonian reductions, which we use to define the elliptic analogue of hypertoric varieties. %We compare our definition with what would be the elliptic coulomb branch for Abelian gauge group from BFN construction. Lastly, we study the elliptic hypertoric variety associated to the Tate curve and show that for $q\!= \!1$, we recover the multiplicative hypertoric varieties. 
\end{abstract}
	
	\maketitle
	
	%\tableofcontents
	
\section{Introduction} 

Additive hypertoric varieties, or simply hypertoric varieties $\mathfrak{M}%
^{+}$, are the hyperkahler analogues of toric varieties introduced in \cite{BD00}
as hyperkahler quotient of complex vector spaces by tori. Namely, 
\begin{equation*}
	\mathfrak{M}^{+}=\mathfrak{M}^{+}( \alpha ,\beta) =\mathbb{C}%
	^{2n}\threebar_{(\alpha ,\beta)}\rK_{\mathbb{R}}
\end{equation*}%
with $\rK_{\mathbb{R}}$ being a real subtorus of $\rT_{\mathbb{R}}\cong U(1)^{n}$ acting on $\mathbb{C}^{2n}=\left( T^{\ast }\mathbb{C}%
\right) ^{n}$ by the standard quaternionic action with real and complex
moment maps%
\begin{equation*}
(\mu _{\mathbb{R}},\mu _{\mathbb{C}})	: \left( T^{\ast }\mathbb{C}\right) ^{n}  \rightarrow \mathfrak{k}_{\mathbb{R}}^{\vee }\oplus  \mathfrak{k}^{\vee },
\end{equation*}
%\begin{eqnarray*}
	%\mu _{\mathbb{R}} &:&\left( T^{\ast }\mathbb{C}\right) ^{n}\rightarrow 
	%\mathfrak{k}_{\mathbb{R}}^{\vee }, \\
	%\mu _{\mathbb{C}} &:&\left( T^{\ast }\mathbb{C}\right) ^{n}\rightarrow 
	%\mathfrak{k}^{\vee },
%\end{eqnarray*}%
where $\alpha $ and $\beta $ are real and complex moment map values
respectively. It can also be described as a GIT quotient of the level set $\mu _{\mathbb{C}%
}^{-1}\left( \beta \right) $. 
%\begin{equation*}
	%1\rightarrow K\overset{}{\rightarrow }T\rightarrow A\rightarrow 1.
%\end{equation*}

In recent years, additive hypertoric varieties appear predominantly in the
context of 3-dimensional $\mathcal{N}=4$ SUSY gauge theories. Such a theory
is determined by a pair $\left( G,T^{\ast }\textbf{N}\right) ,$ where $G$ is a
complex reductive group and $\textbf{N}$ is a $G$-representation\footnote{%
	More generally, $T^{\ast }\textbf{N}$ could be replaced by a symplectic $G$%
	-representation.}. Additive hypertoric varieties are the Higgs branches of
such theories when $G$ is abelian. The gauge theory associated to a
general pair $\left( G,T^{\ast }\textbf{N}\right) $ admits two holomorphic symplectic
varieties, called the Higgs branch and the Coulomb branch, which are moduli
spaces of vacua for the theory. The Higgs branch $\mathcal{M}_{H}$ has a
simple description as a the hyperkahler quotient of $T^{\ast} \textbf{N}$ by the
maximal compact subgroup $G_{\mathbb{R}}\subset G$, 
\begin{equation*}
	\mathcal{M}_{H}=T^{\ast }\textbf{N} \threebar G_{\mathbb{R}}.
\end{equation*}%
The Coulomb branch was defined in \cite{Nak15,BFN16} as an affine variety whose
coordinate ring is the equivariant Borel-Moore homology of the variety of triples equipped with a convolution type
product. When $G=\rK$ is abelian, the Coulomb branch is also an additive
hypertoric variety, we have%
\begin{equation*}
	\mathcal{M}_{H}=T^{\ast } \frak{t} \threebar \rK_{\mathbb{R}}\text{
		and }\mathcal{M}_{C}=T^{\ast } \frak{t}^{\vee} \threebar \rA_{%
		\mathbb{R}}^{\vee },
\end{equation*}
where $\rK_{\mathbb{R}}$ and $\rA_{\mathbb{R}}$ are given by a short exact sequence of tori
 \begin{equation*} 
 	1\longrightarrow \rK {\longrightarrow} \rT {\longrightarrow} \rA \longrightarrow 1.
 \end{equation*}

3d mirror symmetry, or symplectic duality, predicts that there is a dual
3-dimensional $\mathcal{N}=4$, SUSY gauge theory which switches the roles
of the Higgs branch and the Coulomb branch, among many other surprising
predictions. Various predictions of 3d mirror symmetry have been verified
for hypertoric varieties \cite{BLPW14,KMP21,MSY20A,MSY20B,SZ20,Zho21,GHM22} etc. This includes the Hikita conjecture
\cite{Hik17} which in its simplest form states that if $X\rightarrow Y$ and $X^{!}\rightarrow Y^{!}$ are
dual pairs of symplectic resolutions, then there is an isomorphism%
\begin{equation*}
	H^{\ast }(
	X)\cong \mathbb{C} [ (Y^!)^{\rA}],
\end{equation*}%
where $\rA$ denotes the maximal torus in the Hamiltonian automorphism group of 
$X$. For hypertoric varieties and Springer resolutions, the equivariant and quantum extension of Hikita's conjecture was proved in \cite{KMP21} relating the quantized coordinate ring of $Y^!$ with the quantum cohomology of $X$. On the other hand, a K-theoretic analogue of \cite{KMP21} was proved in \cite{Zho21}, which relates the quantum K-theory of $\mathcal{M}_{H}$ with the quantized coordinate ring of the $K$-theoretic Coulomb branch $\mathcal{M}_{C}^{\times}$. For abelian gauge theories, $\mathcal{M}_{C}^{\times}$ is a multiplicative hypertoric variety. (The definitions of additive and multiplicative hypertoric varieties are recalled in Appendix \ref{sec:appendix}.)

%There is a K-theory analog of the Hikita conjecture says that%
%\begin{equation*}
%	\mathbb{C}\left[ \left( Y_{\times }\right) ^{A}\right] \ K^{\ast
	%}\left( X^{!}\right) 
%\end{equation*}%
%where $Y_{\times }$ is the multiplicative analog to $Y$. It is proved in
%[ref] for the hypertoric cases. 

%Multiplicative hypertoric varieties are defined using quasi-Hamiltonian
%manifolds [ref] for which the moment maps take values in the Lie groups.
%This construction ......Explicitly $\mathfrak{M}^{\times }=\mu _{\times
%}^{-1}\left( \beta \right) /K$ with $\mu _{\times }:\left( T^{\ast }\mathbb{C%
%}^{\times }\right) ^{n}\rightarrow K$ is a multiplicative moment map.

It is a natural question whether there exists an analogous result for the elliptic cohomology of $X^{!}$, at
least in the hypertoric cases. The purpose of this paper is twofold. First we give a definition of elliptic
hypertoric varieties $\mathfrak{M}^{\tau}$ and describe their basic
properties. Second we establish the following elliptic version of the equivariant Hikita conjecture: %(note that we only consider the Hamiltonian torus-equivariant version here)
\begin{theorem} (Collorary \ref{cor:Hikita})
	There is an isomorphism
	\begin{equation*}
		\mathrm{Ell}_{\rK^{\vee}}(\mathfrak{M}^+_v)\cong (\widetilde{\mathfrak{M}_u^{\tau}})^{\rA}.
	\end{equation*}
\end{theorem}

Here the subscripts $u$ and $v$ are combinatorial inputs for a pair of dual hypertoric varieties and   
$\widetilde{\mathfrak{M}^{\tau}_u}$ is a deformation space of the elliptic hypertoric variety $\mathfrak{M}^{\tau}_u$. See Section \ref{sec:Hikita} for the details.
%prove that%
%\begin{equation*}
	%\mathbb{C}\left[ \left( Y_{\tau }\right) ^{A}\right] \simeq Ell\left(
	%X^{!}\right) 
%\end{equation*}%
%in the hypertoric cases.

\bigskip 

%Recall that any one dimensional connected complex Lie group is either $%
%\mathbb{C}$, $\mathbb{C}^{\times }$ or an elliptic curve $E_{\tau }=\mathbb{C%
%}^{\times }/q^{\mathbb{Z}}$ with $q=\exp \left( 2\pi i\tau \right) $ and its
%corresponding multiplicative generalized cohomology theory is cohomology $%
%H^{\ast }\left( X\right) $, K-theory $K^{\ast }\left( X\right) $ or elliptic
%cohomlogy $Ell\left( X\right) $ respectively. 

An additive (resp. multiplicative) hypertoric variety is defined as the
quotient of the $\rT\cong \left( \mathbb{C}^{\times }\right) ^{n}$%
-Hamiltonian (resp. quasi-Hamiltonian) holomorphic symplectic space $\left(
T^{\ast }\mathbb{C}\right) ^{n}$ (resp. $(T^*\C\setminus \{zw\! =\! 1\})^n$) with moment map $\mu _{+}:\left( T^{\ast }\mathbb{C}\right)
^{n}\rightarrow \mathfrak{t}^{\vee}\cong  \mathfrak{t}$ (resp. $\mu
_{\times }:(T^*\C\setminus \{zw\! =\! 1\})^n\rightarrow \rT$)
by a subtorus $\rK$.  In order to define elliptic hypertoric varieties, we first need to construct a two dimensional symplectic manifold $X_{\vartheta }$,
as the elliptic analogue to $T^{\ast }\mathbb{C}$ and $T^*\C\setminus \{zw\! =\! 1\}$, which admits a $\mathbb{C}^{\times }$-symmetry and a morphism%
\begin{equation*}
	\mu :X_{\vartheta }\rightarrow E_{\tau }, 
\end{equation*}%
which serves the purpose of a moment map. %In Section \ref{sec:e_ham}, we introduce a general notion of ``elliptic''-Hamiltonian manifolds for which the moment maps are valued in elliptic curves. 
The definition of the building block $X_{\vartheta }$ is given in Section \ref{sec:X_theta}. The elliptic hypertoric varieties $\mathfrak{M}^{\tau }$ are defined in Section \ref{sec:ell_hypertoric} as the ``elliptic''-Hamiltonian reduction (which we introduce in Section \ref{sec:e_ham}) of $X_{\vartheta }^{n}$ by $\rK$. %In Section \ref{sec:}, we study the basic topological and geometric properties of . 
The precise statement and  the proof of elliptic Hikita conjecture in the hypertoric cases will be
given in Section \ref{sec:Hikita}.

\addtocontents{toc}{\protect\setcounter{tocdepth}{1}}

\subsection*{Acknowledgment}
Both authors are very grateful to Michael McBreen for excellent explanations about 3d mirror symmetry, multiplicative hypertoric varieties and Hikita conjecture in numerous discussion sessions.
We would also like to thank Siu-Cheong Lau and Ziming Ma for helpful discussions. The work of N. C. Leung described in this paper was substantially supported by grants from the Research Grants Council of the Hong Kong Special Administrative Region, China (Project No. CUHK14301619 and CUHK14306720) and a direct grant from CUHK (Project No. CUHK4053400).

\section{e-Hamiltonian manifolds and reduction} \label{sec:e_ham}
In this section, we introduce a notion of ``elliptic''-Hamiltonian (e-Hamiltonian) $\rT$-manifolds for which the moment maps are valued in elliptic curves, an important class of complex two-dimensional examples, and the analogue of Hamiltonian reduction. %all of which we use to define elliptic hypertoric varieties in Section \ref{sec:ell}.

\subsection{Moment maps valued in elliptic curves}
Let $\bm{H}$ denote the upper-half plane
\[
\bm{H}=\{\tau\in \C | \Im \tau >0\}.
\]
For $\tau\in \bm{H}$, let $\Gamma=\langle 1,\tau \rangle$  %\langle 1,\tau \rangle$ be 
be the lattice in $\C$ generated by $1$ and $\tau$, and let $E_{\tau}$ denote the elliptic curve
\[
E_{\tau}:=\C/\Gamma=\C^{\times}/q^{\Z}, \qquad q=\exp(2\pi i \tau).
\] 
%where $q=\exp(2\pi i \tau)$.

Suppose $\rT\!=\!(\C^{\times})^n$ acts on a holomorphic symplectic manifold $(X,\omega)$ preserving the symplectic form. We will identify $\frak{t}=\mathrm{Lie}(\rT)$ with $\frak{t}^{\vee}$ via an invariant inner product $\langle -,- \rangle$ on $\frak{t}$. For $\zeta\in\frak{t}$, we denote by $v_{\zeta}$ the vector field on $X$ generated by $\zeta$. Then, the $\rT$-action is Hamiltonian if there is a $\rT$-invariant holomorphic map $\mu_+: X\to \frak{t}$ satisfying
\[
\omega(v_{\zeta},-)=d\langle\mu_+,\zeta\rangle.
\]
On the other hand, the $\rT$-action is ``quasi''-Hamiltonian (cf. \cite{AMM98}), if it is there is a $\rT$-invariant holomorphic map $\mu_{\times}: X\to \rT$ satisfying
\[
\omega(v_{\zeta},-)= \langle \mu_{\times}^{*}\eta, \zeta \rangle,
\]
where 
\[
\eta=\sum_{i=1}^n d\log t_i, \quad (t_1,\ldots,t_n)\in \rT,
\]
is understood as a $\frak{t}$-valued one-form on $\rT$.

Since $\eta$ is invariant under the $q^{\Z^n}$-action on $\rT$ defined by
\[
(q^{k_1},\ldots,q^{k_n})\cdot(t_1,\ldots,t_n)=(q^{k_1}t_1,\ldots,q^{k_n}t_n), 
\]
where $(q^{k_1},\ldots,q^{k_n})\in q^{\Z^n}$, it descends to a $\frak{t}$-valued one form $\eta$ (via an abuse of notation) on $E_{\tau}^n=\rT/q^{\Z^n}$. This leads us to the following definition:

\begin{defn}%[$e$-Hamiltonian $\rT$-manifolds] 
\label{defn:elliptic_T_manifold}
	A symplectic $\rT$-manifold $(X,\omega)$ is called $e$-Hamiltonian if %it is equipepd with an invariant holomorphic symplectic form $\omega$ and 
	there exists an invariant holomorphic map 
	$\mu: M\to E_{\tau}^n$ satisfying
	\begin{equation} \label{eq:ell_moment_map}
	\omega(v_{\zeta},-)= \langle \mu^{*}\eta, \zeta \rangle,
	\end{equation}
	where $\eta\in \Omega^1(E_{\tau}^n,\frak{t})$ is the $\frak{t}$-valued one-form defined above.
	%where $\sum_{i=1}^n \zeta_i dx_i \in \Omega^1(E_{\tau}^n,\frak{t})$, and $dx_i$ is the invariant differential on the $i$-th component of $E_{\tau}^n$. 
\end{defn}

Note that the e-moment map $\mu$ is a part of the definition. For brevity, we will simply refer to $(X,\omega)$ as an e-Hamiltonian $\rT$-manifold or say that the $\rT$-action on $(X,\omega)$ is e-Hamiltonian. %provided the existence of a moment map is understood. 

%\begin{remark}
	%For a connected complex reductive group $G$, $G$ (resp. $\frak{g}$) has the interpretation as the moduli space $G_{E,0}$ of degree $0$, semi-stable $G$-bundles over the nodal (resp. cuspidal) curve $E$ together with a framing at a basepoint (see \cite{FGL20} and \cite{FM01}). When $G=\rT$ and $E=E_{\tau}$, we have $G_{E,0}=E_{\tau}^n$, which is exactly the target space of our e-moment maps. It is therefore natural for e-moment maps to takes value in $G_{E,0}$ when $G$ is nonabelian. We will pursuit this direction in a separate work and focus on hypertoric varieties in this paper.
%\end{remark}

\subsection{Two dimensional examples} \label{sec:examples}
In this subsection we construct a class of complex two dimensional examples $(X_{\theta},\omega_{\theta})$ of e-Hamiltonian $\rT$-manifolds for $\rT=\C^{\times}$. These examples depend on a global section $\theta$ of holomorphic line bundle $L$ over $E_{\tau}$ and a factorization $L=L_1\otimes L_2$. Topologically, $X_{\theta}$ is the total space of a $\C^{\times}$-fibration over $E_{\tau}$ with fibers degenerating to the union $\C\cup_0 \C$ of two affine lines over the vanishing loci of $\theta$. For a detailed construction of $(X_{\theta},\omega_{\theta})$, let's first recall a few basic facts about holomorphic line bundles over elliptic curves and theta functions (cf. \cite{POL03}).

Let $H:\C^2 \to \C$ be a hermitian form\footnote{$H$ is an $\R$-bilinear form which is $\C$-linear in the first argument and $\C$-antilinear in the second argument.} on $\C$ such that $E:=\Im H$ is integral on $\Gamma^2$. Let $\chi:\Gamma\to U(1)$ be a character twisted by $H$, namely $\chi$ satisfies 
\[
\chi(\gamma+\gamma')=\chi(\gamma)\chi(\gamma')\exp(i\pi  E(\gamma,\gamma')), \qquad \gamma,\gamma'\in \Gamma.
\]
To a pair $(H,\chi)$, we can associate a holomorphic line bundle $L_{(H,\chi)}$ over $E_{\tau}$ whose total space is the quotient of $\C^2$ by the $\Gamma$-action defined by
\[
\gamma\cdot(w,x)=\left(e_{(\chi,H)}(\gamma,x) w   ,x+\gamma  \right),
\]
where $e_{(H,\chi)}:\Gamma\times \C\to \C^{\times}$ is the automorphy factor  
\[
e_{(H,\chi)}(\gamma,x)=\chi(\gamma)\exp\left(\pi H(x,\gamma)+\frac{\pi}{2}H(\gamma,\gamma) \right).
\]
More explicitly, for $\gamma=a+b\tau$, we can take 
\begin{equation} \label{eq:H}
H(x,y)=k\cfrac{x\bar{y}}{\Im \tau},
\end{equation}
and
\begin{equation} \label{eq:chi}
\chi(\gamma)=(-1)^{ab}\exp(2\pi i E(x_0,\gamma))
\end{equation}
where $k=\deg L_{(H,\chi)}$ and $x_0\in \C$. 
We denote by $\mathcal{P}$ to set of all pairs $(H,\chi)$ of the form \eqref{eq:H}, \eqref{eq:chi}. It has an obvious group structure with the multiplication given by
$(H,\chi)\cdot (H',\chi')=(H+H',\chi\chi')$. By the Appell-Humbert theorem, the map $(H,\chi)\mapsto L_{(H,\chi)}$ is an isomorphism $\mathcal{P}\cong \Pic(E_{\tau})$.

For a line bundle $L$ over $E_{\tau}$ given by an automorphy factor $e_L$, the global sections of $L$ are holomorphic functions on $\C$ satisfying 
\[
\theta(x+\gamma)=e_{L}(x,\gamma)\theta(x),
\]
which are exactly theta functions with automorphy factor $e_L$. Moreover, we have
\[
\dim H^0(E_{\tau},L)= \begin{cases}
0 & \text{for } \deg L<0\\
1 & \text{for } \deg L=0\\
\deg L & \text{for } \deg L>0
\end{cases}
\] 
by the Riemann-Roch theorem.

Now, let's fix a line bundle $L$ over $E_{\tau}$ with $\deg \cL> 0$, a nonzero global section $\theta \in H^0(E_{\tau},L)$, and a factorization $L=L_1\otimes L_2$. For simplicity, we will assume their  automorphy factors $e_{L}$, $e_{L_1}$, $e_{L_2}$ are in $\mathcal{P}$. The
total space of the vector bundle $L_1\oplus L_2$ over $E_{\tau}$ can be obtained as 
the quotient of $\C^3$ by the action of $\Gamma$ defined as
\begin{equation} \label{eq:gamma_action}
\gamma\cdot (z,w,x)=(e_{L_1}(\gamma,x)z,e_{L_2}(\gamma,x)w,x+\gamma), 
\end{equation}
where $(z,w,x)$ are coordinates on $\C^3$.
%We will assume $\deg (L_1\otimes L_2)\ge 0$.  For a nonzero section $\theta\in H^0(E_{\tau},L_1\otimes L_2)$

We define a hypersurface $X_{\theta}$ in $\mathrm{Tot}(L_1\oplus L_2)$ by 
\[
X_{\theta}=\widetilde{X}_{\theta}/\Gamma,
\]
where $\widetilde{X}_{\theta}$ is the $\Gamma$-invariant hypersurface in $\C^3$ given by 
\begin{equation} \label{eq:X_tilde}
\widetilde{X}_{\theta}=\{(z,w,x)\in \C^3|zw=\theta(x)\}. 
\end{equation}
We note that $X_{\theta}$ has at worst orbifold singularities. We will assume $X_{\theta}$ is smooth, this amounts to $\theta$ having no repeated zeroes, which is the case for generic choices of $\theta$. We equip $X_{\theta}$ with the symplectic form $\omega_{\theta}$ given by the $\Gamma$-invariant symplectic form $\widetilde{\omega}_{\theta}$ on 
$\widetilde{X}_{\theta}$
\begin{equation} \label{eq:omega_tilde}
\widetilde{\omega}_{\theta}=\Res \cfrac{dz\wedge dw\wedge dx}{zw-\theta(x)}.
\end{equation}
The $\C^{\times}$-action on $\widetilde{X}_{\theta}$ defined by
\begin{equation} \label{eq:action}
t\cdot (z,w,x)=(tz,t^{-1}w,x), \quad t\in \C^{\times}.
\end{equation} 
is Hamiltonian with respect to $\widetilde{\omega}_{\theta}$ and commutes with the $\Gamma$-action. The moment map $\widetilde{\mu}_{\theta}:\widetilde{X}_{\theta} \to \C$ satisfies
\[
d\widetilde{\mu}_{\theta} = \Res d\log(zw-\theta(x))\wedge dx = dx.
\]
%where $V=z\frac{\partial}{\partial z}-w\frac{\partial}{\partial w}$. 
Thus, $\widetilde{\mu}_{\theta}$ is the projection to the $x$-plane (up to an additive constant, which we choose to be zero). 
By construction, \eqref{eq:action} defines a $\C^{\times}$-action on $X_{\theta}$ which preserves $\omega_{\theta}$, and $\widetilde{\mu}_{\theta}$ descends to the projection map $\mu_{\theta}:X_{\theta}\to E_{\tau}$.  It is easy to see that $(X_{\theta},\omega_{\theta})$ %is independent of the choice of automorphy factors for $L_1$ and $L_2$ and
satisfies the definition of an e-Hamiltonian $\C^{\times}$-manifold and $\mu_{\theta,L}$ is the e-moment map. We note swapping $L_1$ and $L_2$ in the above construction will result in the same holomorphic symplectic manifold with the opposite $\C^{\times}$-action. 

%\begin{remark}
%The construction of $X_{\theta}$ is independent of the choice of $L_1$. To see this, suppose $L_1'$ is another representative of $\lambda$ with automorphy factor $e_{L'_1}$. Then, we must have 
%\[
%\cfrac{e_{L_1}(\gamma,x)}{e_{L'_1}(\gamma,x)}=\cfrac{h(x+\gamma)}{h(x)}
%\]
%for some holomorphic function $h:\C\to \C^{\times}$. 
%\end{remark}

Implicit in the above construction is a more general recipe for producing e-Hamiltonian $\rT$-manifolds from Hamiltonian $\rT$-manifolds for $\rT\!=\!(\C^{\times})^n$, 
i.e.

\begin{prop} \label{prop:X_tilde}
	Let $(\widetilde{X}, \widetilde{\omega}) $ be a Hamiltomnian $\rT$-manifold.  %$T=(\C^{\times})^n$. 
	Suppose we have a holomorphic, free, proper action of $\Gamma^n$ on $\widetilde{X}$ which commute with the $\rT$-action and preserves $\widetilde{\omega}$.
	 Suppose further that the moment map $\widetilde{\mu}:\widetilde{X}\to \C^n$ intertwines the $\Gamma^n$-action on $\widetilde{X}$ and the natural $\Gamma^n$-action on $\C^n$. Then, $X=\widetilde{X}/\Gamma^n$ equipped with the symplectic from $\omega$ induced from $\widetilde{\omega}$ is an e-Hamiltonian $\rT$-manifold, and $\widetilde{\mu}$ descends to an e-moment map $\mu:X\to E_{\tau}^n$.
\end{prop}
\begin{proof}
Since $(X,\omega)$ is a symplectic $\rT$-manifold, it suffices to check that $\mu$ is an e-moment map, i.e., it satisfies \eqref{eq:ell_moment_map}. In coordinates, we have
	\[
	\widetilde{\omega}(\widetilde{v}_{\zeta},-)=\widetilde{\mu}^*\left (\sum_{i=1}^n \zeta_i dx_i \right).
	\]
where $\widetilde{v}_{\zeta}$ is the vector field on $\widetilde{X}$ generated by $\zeta\in\frak{t}$, and $x_1,\ldots,x_n$ are coordinates on $\frak{t}=\C^n$. Since the actions of $\rT$ and $\Gamma^n$ commute, $\widetilde{v}_{\zeta}$ descends to vector field $v_{\zeta}$ on $X$ generated by $\zeta$. Moreover, $\sum_{i=1}^n \zeta_i dx_i$ is $\Gamma^n$-invariant and defines a $1$-form on $E_{\tau}^n$. 
This means $\mu$ satisfies
\[
\omega(v_{\zeta},-)=\mu^*\left(\sum_{i=1}^n \zeta_i dx_i\right),
\]
which is exactly \eqref{eq:ell_moment_map} since we have 
\[
\sum_{i=1}^n \zeta_i dx_i=\sum_{i=1}^n \zeta_i d\log t_i=\langle \eta,\zeta\rangle, \qquad (t_1,\ldots,t_n)\in \rT.
\] 
\end{proof}

Similarly, we can construct e-Hamiltonian $\rT$-manifolds from quasi-Hamiltonian $\rT$-manifolds:
\begin{prop} \label{prop:X_hat}
	Let $(\hat{X}, \hat{\omega}) $ be a quasi-Hamiltomnian $\rT$-manifold. %$T=(\C^{\times})^n$.
	Suppose we have a holomorphic, free, proper action of $\Z^n$ on $\hat{X}$ which commute with the $\rT$-action and preserves $\hat{\omega}$. Suppose further that the moment map $\hat{\mu}:\hat{X}\to \rT\cong (\C^{\times})^n$ intertwines the $\Z^n$-action on $\hat{X}$ and the natural $q^{\Z^n}$-action on $\rT$. Then, $X=\hat{X}/\Z^n$ equipped with the symplectic from $\omega$ induced from $\hat{\omega}$ is an e-Hamiltonian $\rT$-manifold, and $\hat{\mu}$ descends to an e-moment map $\mu:X\to E_{\tau}^n$.
\end{prop}

\subsection{e-Hamiltonian reduction} 
We now introduce the analogue of Hamiltonian reduction for e-Hamiltonian $\rT$-manifolds. This provides a general method for constructing new e-Hamiltonian $\rT$-manifolds from the existing ones.

Let $\rK\cong (\C^{\times})^k$ be a $k$-dimensional subtorus of $\rT$. The inclusion $\phi:\rK\to \rT$ has the form
\begin{equation} \label{eq:K}
\phi(a_1,\ldots,a_k)=\left(\prod_{j=1}^k a_j^{\phi_{1j} },\ldots, \prod_{j=1}^k a_j^{\phi_{nj}}\right),
\end{equation}
where $(\phi_{ij})\in \mathrm{Mat}_{n\times k}(\Z)$. %are entries of an integer-valued $n$ by $\rk$ matrix $(\phi_{ij})$.   
We defines a morphism $\phi^{\vee}_{\tau}: E_{\tau}^n \to E_{\tau}^k$ by
\begin{equation} \label{eq:phi_vee}
\phi^{\vee}_{\tau}([x_1,\ldots,x_n])=\left( \left[\sum_{i=1}^n \phi_{i1}x_i,\ldots,\sum_{i=1}^n \phi_{ik}x_i \right] \right).  
\end{equation}

The following lemma shows that the $\rK$-action on an e-Hamiltonian %$-manifold is automatically e-Hamiltonian. This is similar to the  quasi-Hamiltonian case \cite{GAN18,COO16}.

\begin{lemma} \label{lemma:K_action}
	Suppose $(X,\omega)$ is an e-Hamiltonian $\rT$-manifold with e-moment map $\mu:X\to E_{\tau}^n$. Then it is also an e-Hamiltonian $\rK$-manifold with e-moment map $\mu_{\rK}=\phi^{\vee}_{\tau}\circ \mu:X\to E_{\tau}^k$ for any subtorus $\rK$ in $\rT$.
\end{lemma}
\begin{proof}
Since $(X,\omega)$ is a symplectic K-manifold, it suffices to check that $\mu_{\rK}$ is an e-moment map, i.e, it satisfies \eqref{eq:ell_moment_map}. Let $\eta_T$ and (resp. $\eta_K$) be denote the $\frak{t}$- (resp. $\frak{k}$-) valued $1$-forms on $E_{\tau}^n$ (resp. $E_{\tau}^k$) as in Definition \ref{defn:elliptic_T_manifold}. Let $\iota: \frak{k}\to \frak{t}$ and $\iota^{\vee}: \frak{t}\to \frak{k}$ denote both the maps on Lie algebras and the maps on the the Lie algebra-valued differential forms $\iota: \Omega(-,\frak{k})\to  \Omega(-,\frak{t})$ and $\iota^{\vee}: \Omega(-,\frak{t})\to \Omega(-,\frak{k})$ induced by $\phi$. Then, we have
\[
\langle \mu_{\rK}^*\eta_{\rK}, \zeta\rangle= \langle \mu^*\circ (\phi^{\vee}_{\tau})^*\eta_{\rK},\zeta \rangle = \langle  \iota^{\vee}(\mu^*\eta_{\rT}),\zeta)\rangle=\langle \mu^*\eta_{\rT}, \iota(\zeta) \rangle=\omega(v_{\zeta},-),
\]
for all $\zeta\in \frak{k}$.

\end{proof}

We now define the analogue of Hamiltonian reduction for e-Hamiltonian $\rT$-manifolds. % which we call e-Hamiltonian reduction.

\begin{defn} %[e-Hamiltonian reduction] 
	\label{def:e_reduction}
	Suppose $(X,\omega)$ is an e-Hamiltonian $\rK$-manifold, with e-moment map $\mu_{\rK}:X\to E_{\tau}^k$. Let $(\cL,\beta)$ be a pair such that $\beta \in E_{\tau}^k$ and $\cL$ is a $\rK$-linearized ample line bundle over $X$. We define an e-Hamiltonian reduction $M(\cL,\beta)$ of $X$ by $\rK$ to be the GIT quotient 
	\[
	M(\cL,\beta)=\mu_{\rK}^{-1}(\beta)\twobar_{\cL} K.
	\] 
	%Here the character $\alpha$ linearizes the $\rk$-action on $L$. 
\end{defn}

%The pair $(\cL,\beta)$ is called generic if the semi-stable locus of $\mu_{\rK}^{-1}(\beta)$ coincides with the stable locus, in which case $M(\cL,\beta)$ has at worst orbifold singularities. For a generic pair $(\cL,\beta)$, if the the matrix $(\phi_{ij})$ is unimodular, then $\rk$ acts freely on the stable locus, in which case $M(\cL,\beta)$ is a complex manifold.

Let's assume $M(\cL,\beta)$ is a complex manifold and denote by $\bar{\omega}$ the induced symplectic form. The restriction of $\mu$ to $\mu_{\rK}^{-1}(\beta)$ is $\rK$-invariant and descends to a map $\bar{\mu}: M(\cL,\beta)\to E_{\tau}^n$. 
It is easy to see that the residual $\rT$-action on $(M(\cL,\beta),\bar{\omega})$ is e-Hamiltonian and $\bar{\mu}$ is an e-moment map. On the other hand, the subtorus $\rK$ acts trivially on $M(\cL,\beta)$, so we have an action of $\rA=\rT/\rK$ on $M(\cL,\beta)$. The following lemma shows that the $\rA$-action on $M(\cL,\beta)$ is also e-Hamiltonian.

\begin{lemma} \label{lemma:residual_action}
Let $(X,\omega)$ be an e-Hamiltonian $\rT$-manifold with an e-moment map $\mu:X\to E_{\tau}^n$ and let $\rK\subset \rT$ be a subtorus which acts trivially on $X$. Then the action of $\rA=\rT/\rK$ on $X$ is also e-Hamiltonian. %with e-moment map $\mu_{\rA}$. 
\end{lemma}
\begin{proof}
Since $(X,\omega)$ is a symplectic $\rA$-manifold, it suffices to show the existence of an e-moment map. Let $\pi: \frak{t}\to \frak{a}$ and $\pi^{\vee}: \frak{a}\to \frak{t}$ denote both the maps on Lie algebras and the maps on the the Lie algebra-valued differential forms $\pi: \Omega(-,\frak{t})\to  \Omega(-,\frak{a})$ and $\pi^{\vee}: \Omega(-,\frak{a})\to \Omega(-,\frak{t})$ induced by the quotient map $\psi:\rT\to \rA$. We have an short exact sequence
\begin{equation} \label{eq:ell_ses}
1\longrightarrow E_{\tau}^d\overset{\psi^{\vee}_{\tau}}{\longrightarrow} E_{\tau}^n \overset{\phi^{\vee}_{\tau}}{\longrightarrow} E_{\tau}^k \longrightarrow 1,
\end{equation}
where $d=n-k$. For $a\in E_{\tau}^n$ we denote by $t_{a}:E_{\tau}^n\to E_{\tau}^n$ the translation by $a$. Since the $\rK$-action on $X$ is trivial, the moment map $\mu_{\rK}=\phi^{\vee}_{\tau}\circ \mu$ is necessarily constant. 
We fix $a_0\in \Image{\mu}$ and define a map $\mu_{\rA}:X\to E_{\tau}^d$ by
\begin{equation} \label{eq:D_moment_map}
\mu_{\rA}=(\psi^{\vee}_{\tau})^{-1}\circ t_{a_0^{-1}}\circ \mu. 
\end{equation} 
Note that $(\psi^{\vee}_{\tau})^{-1}$ is an isomorphism $\ker \phi^{\vee}_{\tau}\isoto E_{\tau}^d$. Then, the map $\mu_{\rA}$ satisfies
\[
\langle \mu_{\rA}^*\eta_{\rA}, \pi(\zeta)\rangle=\langle \mu_{\rA}^* (\pi^{\vee} (\eta_{\rA})),\zeta\rangle=\langle \mu_{\rA}^* \circ (\psi^{\vee}_{\tau})^*\eta_{\rT},\zeta\rangle =\langle \mu^*\eta_{\rT},\zeta\rangle=\omega(v_{\zeta},-)=\omega(v_{\pi(\zeta)},-), 
\]
for all $\zeta\in \frak{t}$, where $\eta_{\rA}$ denotes the $\frak{a}$-valued one form on $E_{\tau}^d$ defined as in Definition \ref{defn:elliptic_T_manifold}. %The last equality is due to the fact that the $\rk$-action on $X$ is trivial. 
Since the map $\pi:\frak{t}\to \frak{a}$ is surjective, it follows that $\mu_{\rA}$ is an e-moment map.
\end{proof}

\section{Elliptic hypertoric varieties} \label{sec:ell}

\subsection{The building block} \label{sec:X_theta}
Additive and multiplicative hypertoric varieties are given by Hamiltonian and quasi-Hamiltonian reduction of $(T^*\C)^n$ and $(T^*\C\setminus \{zw\!=\!1\})^n$ by a subtorus $\rK$ of $\rT\!=\!(\C^{\times})^n$, respectively (see appendix \ref{sec:appendix}). Analogously, we define elliptic hypertoric varieties to be e-Hamiltonian reduction of $X_{\vartheta}^n$ by $\rK$. Here the building block $X_{\vartheta}$ belongs to the class of two dimensional e-Hamiltonian $\C^{\times}$-manifold introduced in Section \ref{sec:examples}. A detailed description of $X_{\vartheta}$ is as follows.

Let $\tau\in \bm{H}$ and $\Gamma=\langle 1,\tau\rangle$. We set 
\[
E_{\tau}=\C/\Gamma=\C^{\times}/q^{\Z}, \quad  q=\exp(2\pi i \tau). 
\]
Let $\widetilde{X}_{\vartheta}$ be a hypersurface in $\C^3$ defined by
\[
\widetilde{X}_{\vartheta} =\{(z,w,x)\in \C^3|zw=\vartheta(x)\}. 
\] 
Here $\vartheta:\C\to \C$ is the Jacobi theta function 
\begin{equation*}
\theta(x)=(t^{1/2}-t^{-1/2}) \prod_{m>0}(1-q^mt)(1-q^mt^{-1}), \quad  t=\exp(2\pi i x),
\end{equation*}
which satisfies
\begin{equation*}
\vartheta(x+\gamma)=(-1)^b q^{-b^2/2}t^{-b} \vartheta(x), \quad \gamma=a+b\tau\in \Gamma.
\end{equation*}
%and
%\[
%\vartheta(-x)=-\vartheta(x)
%\]
The theta function $\vartheta$ defines a section of a degree $1$ line bundle over $E_{\tau}$ with a unique zero at the identity element $e$. We denote this line bundle by $\cO(e)$ and note that it has the automorphy factor %$\alpha:\Gamma\times \C\to \C^{\times}$, 
\begin{equation*}
\alpha:\Gamma\times \C\to \C^{\times}, \qquad \alpha(\gamma,x)= (-1)^b q^{-b^2/2}t^{-b}. 
\end{equation*}
Let $\Gamma$ act on $\C^3$ by
\[
\gamma\cdot (z,w,x)=(z,\alpha(\gamma,x)w,x+\gamma). 
\]
The quotient $\C^3/\Gamma$ is the total space $\mathrm{Tot}(\cO\oplus\cO(e))$ of the rank $2$ bundle $\cO\oplus\cO(e)$ over $E_{\tau}$. 

We define $X_{\vartheta}$  to be the hypersurface 
\[
X_{\vartheta}=\widetilde{X}_{\vartheta}/\Gamma 
\]
in $\mathrm{Tot}(\cO\oplus\cO(e))$. 
We equip $\widetilde{X}$ with the holomorphic sympletic form  $\widetilde{\omega}_{\vartheta}$  \[
\widetilde{\omega}_{\vartheta}=\Res \cfrac{dz\wedge dw\wedge dx}{zw-\vartheta(x)}.
\] 
This in turn equips $X_{\vartheta}$ with a holomorphic symplectic form $\omega_{\vartheta}$.
%We note that $X_{\vartheta}$ is the total space of a $\C^{\times}$-fibration over $E_{\tau}$ with fibers degenerating to the union $\C\cup_0 \C$ of two affine lines over $e$. 
The $\C^{\times}$-action defined by \eqref{eq:action} on $(X_{\vartheta},\omega_{\vartheta})$ is e-Hamiltonian with the e-moment map given by the projection to $E_{\tau}$. The e-Hamiltonian $\C^{\times}$-manifold $(X_{\vartheta},\omega_{\vartheta})$ will be our building block for elliptic hypertoric varieties.

\subsection{Hypertoric varieties via e-Hamiltonian reduction}  \label{sec:ell_hypertoric}
Let $\mathfrak{t}$ and $\mathfrak{a}$ be complex vector spaces of dimensions $n$ and $d$ respectively. Let $\{e_1,\ldots,e_n\}\subset \mathfrak{t}_{\Z}$ be an integer basis and $\{e_1^{\vee},\ldots,e_n^{\vee}\}$ the dual basis. The combinatorial input for a hypertoric variety is a collection $u=\{u_1,\ldots,u_n\}$ of primitive vectors in $\mathfrak{a}_{\Z}$ which span $\mathfrak{a}_{\Z}$ over the integers. From which we obtain short exact sequences
\begin{equation} \label{eq:ses1}
0\longrightarrow \mathfrak{k}\overset{\iota}{\longrightarrow}\mathfrak{t}\overset{\pi}{\longrightarrow}\mathfrak{a}\longrightarrow 0,
\end{equation}

\begin{equation}  \label{eq:ses2}
0\longleftarrow \mathfrak{k}^{\vee}\overset{\iota^{\vee}}{\longleftarrow}\mathfrak{t}^{\vee}\overset{\pi^{\vee}}{\longleftarrow}\mathfrak{a}^{\vee}\longleftarrow 0.
\end{equation} 
by setting $\pi(e_i)=u_i$. Here $\mathfrak{k}=\ker \pi$, and \eqref{eq:ses2} is the dual sequence of \eqref{eq:ses1}. Exponentiating \eqref{eq:ses1}  gives us a short exact sequence of tori
\begin{equation} \label{eq:ses3}
1\longrightarrow \rK \overset{\phi}{\longrightarrow} \rT \overset{\psi}{\longrightarrow} \rA \longrightarrow 1,
\end{equation}
where $\rT=\mathfrak{t}/\mathfrak{t}_{\Z}$, $\rA=\mathfrak{a}/\mathfrak{a}_{\Z}$, and $\rK=\ker (\rT\to \rA)$. 
%Following Section, we also have an short exact sequence
%\[
%1\longrightarrow E_{\tau}^d\overset{\psi^{\vee}_{\tau}}{\longrightarrow} E_{\tau}^n \overset{\phi^{\vee}_{\tau}}{\longrightarrow} E_{\tau}^k \longrightarrow 1.
%\]

Let's consider $X_{\vartheta}^n$ equipped with the holomorphic sympletic form  $\omega_{\C}$ induced from the $\Gamma^n$-invariant symplectic form 
\begin{equation*} \label{eq:omega_c}
\widetilde{\omega}_{\C}=\sum_{i=1}^n \Res \cfrac{dz_i\wedge dw_i\wedge dx_i}{z_iw_i-\vartheta(x_i,\tau)} 
\end{equation*}
on $\widetilde{X}_{\vartheta}^n$. The $\rT$-action on $\widetilde{X}_{\vartheta}^n$ defined by
\[
\vec{t}\cdot(\vec{z},\vec{w},\vec{x})=(t_1z_1,t_1^{-1}w_1,x_1,\ldots,t_nz_n,t_n^{-1}w_n,x_n)
\]
is Hamiltonian with respect to $\widetilde{\omega}_{\C}$ with the moment map $\widetilde{\mu}:\widetilde{X}_{\vartheta}^n\to \C^n$ given by 
\[
\widetilde{\mu}(\vec{z},\vec{w},\vec{x})=\vec{x}.
\]
%the projections to the coordinates $x_1,\ldots,x_n$. 
This in turn defines an e-Hamiltonian $\rT$-action on $(X_{\vartheta}^n,\omega_{\C})$ with the e-moment map $\mu_{\C}: X_{\vartheta}^n\to E_{\tau}^n$ given by the projection to $E_{\tau}^n$. The $\rK$-action on $(X_{\vartheta}^n,\omega_{\C})$ is also e-Hamiltonian by Lemma \ref{lemma:K_action} with the moment map given by $\mu_{\C,\rK}=\phi^{\vee}_{\tau}\circ \mu_{\C}$, where $\phi^{\vee}_{\tau}: E_{\tau}^n\to E_{\tau}^k$ is the ``transpose'' of $\phi$ defined in \eqref{eq:phi_vee}.

We choose an ample line bundle $\cL=\otimes_{i=1}^n \cL_i$ over $X_{\vartheta}^n$, where $\cL_i$ is the pullback of the line bundle $\cO(e)$ over $E_{\tau}$ via the composition
\[
X_{\vartheta}^n\overset{\mu_{\C}}{\longrightarrow}E_{\tau}^n\overset{\mathrm{pr}_i}{\longrightarrow} E_{\tau}.
\]
Here $\mathrm{pr}_i$ denotes the projection to the $i$-th component. More explicitly, the total space of $\cL$ is the quotient of $\C\times \widetilde{X}_{\vartheta}$ by the $\Gamma^n$-action defined as
\[
\vec{\gamma}\cdot(v, (\vec{z},\vec{w},\vec{x}))=\left(\left(\prod_{i=1}^n \alpha (\gamma_i,x_i)\right)v ,\vec{\gamma}\cdot(\vec{z},\vec{w},\vec{x})\right).  
\] 
We will also equip $X_{\vartheta}^n$ with a Kähler form $\omega_{\R}$ induced by the $\Gamma^n$-invariant Kähler form 
\begin{equation*} \label{eq:omega_r}
%\widetilde{\omega}_{\R}= \cfrac{\sqrt{-1}}{2}\partial\bar{\partial} \left( \sum_{i=1}^n  |z_i|^2+ \exp\left (-2\pi \Im\left (\frac{x_i}{\tau}\right)^2 \right)|w_i|^2+\frac{|x_i|^2}{\Im \tau} \right),
\widetilde{\omega}_{\R}= \cfrac{\sqrt{-1}}{2}\partial\bar{\partial} \left( \sum_{i=1}^n  |z_i|^2+ \exp\left (-2\pi  \frac{(\Im x_i)^2}{\Im \tau} \right)|w_i|^2+\frac{|x_i|^2}{\Im \tau} \right),
\end{equation*}
on $\widetilde{X}_{\vartheta}^n$. Note that $[\omega_{\R}]=c_1(\cL)$. 

The restriction of the $\rT$-action on $\widetilde{X}_{\vartheta}^n$ to the compact real torus $\rK_{\R}\cong U(1)^k$ in $\rK$ is Hamiltonian with respect to $\widetilde{\omega}_{\R}$. The moment map $\widetilde{\mu}_{\R,\rK}:\widetilde{X}^n_{\vartheta}\to \mathfrak{k}_{R}^{\vee}$ given by 
\[
\widetilde{\mu}_{\R,\rK}(\vec{z},\vec{w},\vec{x})=\frac{1}{2} \sum_{i=1}^n 	\left( |z_i|^2 - \exp\left (-2\pi \frac{(\Im x_i)^2}{\Im \tau} \right) |w_i|^2 \right) \iota^{\vee} e_i^{\vee}
\]
is $\Gamma^k$-invariant and descends to the moment map $\mu_{\R,\rK}:X_{\vartheta}^n\to \mathfrak{k}_{R}^{\vee}$ for the Hamiltonian $\rK_{\R}$-action on $(X^n_{\vartheta},\omega_{\R})$.

We are now ready to define elliptic hypertoric varieties.

\begin{defn}[Elliptic hypertoric varieties] 
	\label{defn:ell_hypertoric}
	Given a collection $u=\{u_1,\ldots,u_n\}$ of primitive vectors in $\mathfrak{a}_{\Z}$, $\tau\in\bm{H}$, and 
	a choice of parameters 
	$(\alpha,\beta)\in \mathfrak{k}_{\R}^{\vee}\times E_{\tau}^k$, we define the associated (elliptic) hypertoric variety $\mathfrak{M}^{\tau}_u(\alpha,\beta)$ to be e-Hamiltonian reduction (Definition \ref{def:e_reduction}) of $X_{\vartheta}^n$ by $\rK$
	\begin{equation} \label{eq:ell_hypertoric}
	\mathfrak{M}^{\tau}_u(\alpha,\beta)=(\mu_{\C,\rK})^{-1}(\beta)\twobar_{\alpha} \rK.  %(\mu_{\R}^K,\mu_{\C}^K)^{-1}(\alpha,\beta)/\rK_{\R}. 
	\end{equation}
	Here $\alpha$ determines a lift of the $\rK$-action to the total space of $\cL$.  Equivalently, we can define $\mathfrak{M}^{\tau}_u(\alpha,\beta)$ to be the symplectic reduction of $(\mu_{\C,\rK})^{-1}(\beta)$ by $\rK_{\R}$
	\[
	 \mathfrak{M}^{\tau}_u(\alpha,\beta)=(\mu_{\R,\rK},\mu_{\C,\rK})^{-1}(\alpha,\beta)/\rK_{\R}. 
	\]
\end{defn}
The GIT quotient \eqref{eq:ell_hypertoric} is a priori only defined for characters $\alpha\in\mathfrak{k}_{\Z}=\Hom(\rK,\C^{\times})$. However, as GIT stability condition remains unchanged if we replace $\alpha$ by multiple of itself, we can formally define \eqref{eq:ell_hypertoric}  for $\alpha\in\mathfrak{k}_{\Z}\otimes_{\Z}\Q$. Furthermore, since GIT stability depends locally constantly on $\alpha$, we can then extend \eqref{eq:ell_hypertoric} to $\mathfrak{k}_{\R}$. It is also easy to see that $\mathfrak{M}^{\tau}_u(\alpha,\beta)$ does not depend on the signs of the vectors $u_i$ in the input datum.

Similar to the case of additive hypertoric varieties, $\mathfrak{M}^u_{\tau}(\alpha,\beta)$ is the most singular when $(\alpha,\beta)=(0,0)$, and variations of the parameters $\alpha$ and $\beta$ correspond to partial resolutions and complex deformations of $\mathfrak{M}^u_{\tau}(0,0)$, respectively.

When $\mathfrak{M}^{\tau}_u (\alpha,\beta)$ is smooth (see Theorem \ref{thm:smoothness_ell} for smoothness criteria), it is equipped with a  holomorphic symplectic form $\bar{\omega}_{\C}$ and a Kähler form $\bar{\omega}_{\R}$ induced by $\omega_{\C}$ and $\omega_{\R}$, respectively. 
The residue action of the torus $\rA$ on $\mathfrak{M}^{\tau}_u(\alpha,\beta)$ is e-Hamiltonian with respect to $\bar{\omega}_{\C}$ (by Lemma \ref{lemma:residual_action}) and Hamiltonian with respect to $\bar{\omega}_{\R}$. The corresponding moment maps $\mu_{\R,\rA}: \mathfrak{M}_u^{\tau}(\alpha,\beta) \to \mathfrak{a}_{\R}^{\vee}$ and $\mu_{\C,\rA}: \mathfrak{M}_u^{\tau}(\alpha,\beta) \to E_{\tau}^d$ are given by

\[
\mu_{\R,\rA}([\vec{z},\vec{w},\vec{x}])=\frac{1}{2}\sum_{i=1}^n \left( |z_i|^2 - \exp\left (-2\pi \frac{(\Im x_i)^2}{\Im \tau} \right)|w_i|^2 -\alpha_i \right)  e_i^{\vee} \in \ker (\mathfrak{t}_{\R}^{\vee}\to \mathfrak{k}_{\R}^{\vee})=\frak{a}^{\vee}_{\R},
\]
and
\[
%\mu_{\C,\rA}([\vec{z},\vec{w},\vec{x}])=\sum_{i=1}^n \left[x_ie_i^{\vee} \right]-\beta_i \in \ker(E_{\tau}^n \overset{\phi^{\vee}_{\tau}}{\longrightarrow} E_{\tau}^k)\cong E_{\tau}^d, 
\mu_{\C,\rA}([\vec{z},\vec{w},\vec{x}])=\sum_{i=1}^n \left[x_i-\beta_i \right] \in \ker(E_{\tau}^n \overset{\phi^{\vee}_{\tau}}{\longrightarrow} E_{\tau}^k)\cong E_{\tau}^d, 
\]
where $(\alpha_1,\ldots,\alpha_n)\in \mathfrak{t}_{\R}^{\vee}$ and $(\beta_1,\ldots,\beta_n)\in E_{\tau}^n$ are lifts of $(\alpha,\beta)$. We note that both $\mu_{\R,\rA}$ and $\mu_{\C,\rA}$ are surjective. 

The following are elliptic hypertoric analogues of some familiar spaces:

\begin{example} \label{ex:TPd}
	Let $\{u_1,\ldots,u_{n}\}\subset \mathfrak{a}_{\Z}\cong \Z^{n-1}$ be a collection such that $u_1,\ldots,u_{n-1}$ is an integer basis and $u_{n}=-\sum_{i=1}^{n-1} u_i$. Then $\rK$ is the diagonal subtorus. For $\alpha$ a regular value and $0\in E_{\tau}$ the identity element, $\mathfrak{M}^{\tau}_u(\alpha,0)$ is the analogue of the additive hypertoric variety $\mathfrak{M}_u^+(\alpha,0)=T^*\bP^{n-1}$.
\end{example}

\begin{example} \label{ex:A_n}
	Let $u_1\in\mathfrak{a}\cong \Z$ be a primitive integer vector and define the map $\pi:\mathfrak{t} \to\mathfrak{a}$ by $\pi(e_i)=u_1$  for $i=1,\ldots,n$. Then $\rK$ is the subtorus
	\begin{equation} \label{eq:A_n}
	K=\{(t_1,\ldots,t_{n})\in T| \prod_{i=1}^{n} t_i=1\}.
	\end{equation}
	For $\alpha$ a regular value and $0\in E_{\tau}^{n-1}$ the identity element, $\mathfrak{M}^{\tau}_u(\alpha,0)$ is the analogue of the additive hypertoric variety  $\mathfrak{M}_u^+(\alpha,0)= \widetilde{\C^2/\Z_{n}}$, and the latter is the minimal resolution of the type-$A_{n-1}$ Kleinian singularity.
\end{example}

%For simplicity, we have chosen $X_{\vartheta}$ to be the building block of elliptic hypertoric varieties. As a direct generalization, we can replace $X_{\vartheta}^n$ with $\prod_{i=1}^n X_{\theta_i,L_i}$, where $\theta_i$ is a nonzero global section of a degree one line bundle $\cL_i$ over $E_{\tau}$, $L_i$ is a line bundle over $E_{\tau}$, and $X_{\theta_i,L_i}$ is the hypersurface in $\mathrm{Tot}(L_i\oplus (\cL_i\otimes L_i^{-1})\to E_{\tau})$ defined by \eqref{eq:X_tilde}. 
%The spaces $X_{\theta}$ introduced in Section \ref{sec:examples} with $\theta$ a section of a line bundle of degree greater than $1$ can be constructed as the e-Hamiltonian reduction of $\prod_{i=1}^n X_{\theta_i,L_i}$ by the subtorus $\rK$ \eqref{eq:A_n} for suitable $X_{\theta_i,L_i}$. %In this sense, all these examples are elliptic hypertoric varieties.   

\subsection{Topology and geometry of elliptic hypertoric varieties}
In \cite{DS17}, Dancer and Swann introduced 
additive hypertoric varieties of infinite topological type
as hyperkähler quotients of Hilbert manifolds, generalizing the previous constructions of additive hypertoric varieties of finite topological type \cite{BD00}, and of complex $2$-dimensional hyperkähler manifolds of $A_{\infty}$-type \cite{AKL89,Got94}. In order for the hyperkähler quotient construction to work, the hyperplane arrangement associated to a resulting hypertoric variety cannot be periodic. On the other hand, the universal cover of an elliptic hypertoric variety can be viewed as a version of additive hypertoric variety of infinite topological type whose associated hyperplane arrangement is periodic. %Despite not being a hyperkähler quotient, 
%It shares many properties with additive hypertoric varieties. 
In this subsection, we will study the topology and geometry of elliptic hypertoric varieties through their universal covers.

Let $\widetilde{\mu}_{\R,\rK}:\widetilde{X}^n_{\vartheta}\to \frak{k}_{\R}^{\vee}$ and $\widetilde{\mu}_{\C,\rK}:\widetilde{X}^n_{\vartheta}\to  \frak{k}^{\vee}$ be the moment maps for the actions of $\rK$ and $\rK_{\R}$ on $\widetilde{X}^n_{\vartheta}$ with respect to $\widetilde{\omega}_{\C}$ and $\widetilde{\omega}_{\R}$, respectively. Let $(\alpha,\tilde{\beta})\in \frak{k}_{\R}^{\vee}\times \mathfrak{k}$ %be a lift of $\beta$. 
We define $\widetilde{\mathfrak{M}}_{\tau}^u(\alpha,\tilde{\beta})$ to be the Hamiltonian reduction
\[
\widetilde{\mathfrak{M}}_{\tau}^u(\alpha,\tilde{\beta}):=(\widetilde{\mu}_{\R,\rK},\widetilde{\mu}_{\C,\rK})^{-1} (\alpha,\tilde{\beta})/\rK_{\R}= (\widetilde{\mu}_{\C,\rK})^{-1}(\tilde{\beta})\twobar_{\alpha} \rK.
\]
The moment maps for the residual $\rA$-action on $\widetilde{\mathfrak{M}}^u_{\tau}(\alpha,\tilde{\beta})$ are given by
\[
\widetilde{\mu}_{\R,\rA}(\vec{z},\vec{w},\vec{x})=\frac{1}{2}\sum_{i=1}^n \left( |z_i|^2 - \exp\left (-2\pi \frac{(\Im x_i)^2}{\Im \tau} \right)|w_i|^2 -\alpha_i \right)  e_i^{\vee} \in \ker (\mathfrak{t}_{\R}^{\vee}\to \mathfrak{k}_{\R}^{\vee})=\frak{a}^{\vee}_{\R},
\]
and
\[
\widetilde{\mu}_{\C,\rA}(\vec{z},\vec{w},\vec{x})=\sum_{i=1}^n  (x_i-\widetilde{\beta_i}) e_i^{\vee}  \in \ker(\frak{t}^{\vee} \to \frak{k}^{\vee})\cong \frak{a}^{\vee},
\]
where $(\tilde{\beta}_1,\ldots,\tilde{\beta}_n)\in \frak{t}^{\vee}$ is a lift of $\tilde{\beta}$. 

The sequence \eqref{eq:ses2} induces a short exact sequence of lattices
\begin{equation}
0\to \Gamma^d\to \Gamma^n\to \Gamma^k\to 0. 
\end{equation}
The action of $\Gamma^d$ on $\widetilde{X}^n_{\vartheta}$ commutes with the $\rK$-action and preserves the semi-stable locus of the level set  $(\widetilde{\mu}_{\C,\rK})^{-1}(\tilde{\beta})$ and thus descends to an action on $\widetilde{\mathfrak{M}}^u_{\tau}(\alpha,\tilde{\beta})$. If $\tilde{\beta}$ is a lift of $\beta\in E_{\tau}^k$, then the quotient of $\widetilde{\mathfrak{M}}^u_{\tau}(\alpha,\tilde{\beta})$ by $\Gamma^d$ is the elliptic hypertoric variety $\mathfrak{M}^u_{\tau}(\alpha,\beta)$, and $\widetilde{\mu}_{\C,\rA}$ descends to the e-moment map $\mu_{\C,\rA}:\mathfrak{M}^u_{\tau}(\alpha,\beta)\to E_{\tau}^d$. %By construction, $\widetilde{\mathfrak{M}}^u_{\tau}(\alpha,\beta)$ is a covering space of $\mathfrak{M}^u_{\tau}(\alpha,\beta)$. In fact, it is a universal cover:

%\begin{prop}
%$\pi_1(\widetilde{\mathfrak{M}}^u_{\tau}(\alpha,\beta))=0$. 
%\end{prop}

We now introduce the hyperplane arrangements associated to $\mathfrak{M}^u_{\tau}(\alpha,\beta)$ and $\widetilde{\mathfrak{M}}^u_{\tau}(\alpha,\tilde{\beta})$, which are combinatorial devices that encode the data used to construct hypertoric varieties.

The real and elliptic hyperplane arrangements associated to $\mathfrak{M}^u_{\tau}(\alpha,\beta)$ are collections of hyperplanes $\mathcal{H}_{\R}=\{H_{\R,i}\}_{i=1}^n$ and $\mathcal{H}_{\tau}=\{H_{\tau,i}\}_{i=1}^n$ in $\frak{a}^{\vee}_{\R}$ and $E_{\tau}^d$, respectively, 
\[
H_{\R,i}=\{a\in \frak{a}_{\R}^{\vee}|\langle a, u_i \rangle-\alpha_i =0\},
\]
where $(\alpha_1,\ldots,\alpha_n)\in \frak{t}_{\R}^{\vee}$ is a lift of $\alpha$, and 
\begin{equation} \label{eq:ell_hyperplane}
H_{\tau,i}=\{b=(b_1,\ldots,b_n)\in E_{\tau}^d= \ker(E_{\tau}^n \overset{\phi^{\vee}_{\tau}}{\longrightarrow} E_{\tau}^k) |b_i-\beta_i=0\}. 
\end{equation}
%The hyperplane arrangements associated to the examples \ref{ex:TPd} and \ref{ex:A_n} are depicted in   
%Figure \ref{fig:hyperplane-arrangement}. 
We denote by
\[
\mathcal{H}_{\C}=\{H_{\C,i,\gamma}|i=1,\ldots,n, \gamma \in\Gamma\}
\] 
the periodic complex hyperplane arrangement in $\frak{a}^{\vee}$, %lifting $\mathcal{H}_{\tau}$, 
\[
H_{\C,i,\gamma}=\{\tilde{b}\in \frak{a}^{\vee}= \ker(\frak{t}^{\vee} \to \frak{k}^{\vee})| \langle \tilde{b}, u_i\rangle - \tilde{\beta}_i-\gamma=0\}.
\] 
$\mathcal{H}_{\R}$ and $\mathcal{H}_{\C}$ are the real and complex hyperplane arrangements associated to $\widetilde{\mathfrak{M}}^u_{\tau}(\alpha,\tilde{\beta})$. We will also denote by $\mathcal{A}=\{\mathcal{A}_i\}$ and $\widetilde{\mathcal{A}}=\{\widetilde{\mathcal{A}}_{i,\gamma}\}$ the arrangements of real codimension $3$ subspaces in $\frak{a}^{\vee}_{\R}\times E_{\tau}^d$ and $\frak{a}^{\vee}_{\R}\oplus\frak{a}^{\vee}$, respectively, 
\[
\mathcal{A}_i=H_{\R,i}\times H_{\tau,i},
\]
\[
\widetilde{\mathcal{A}}_{i,\gamma}=H_{\R,i}\times H_{\C,i,\gamma}.
\] 

%\begin{figure}[htb!]
	%\includegraphics[scale=0.5]{hyperplane-arrangement.pdf}
	%\caption{\textcolor{red}{*}placeholder}.
	%\label{fig:hyperplane-arrangement}
%\end{figure}

The following properties of $\widetilde{\mathfrak{M}}^u_{\tau}(\alpha,\tilde{\beta})$ can be characterized in term of the combinatorics of its hyperplane arrangements.

\begin{theorem} {\cite[Theorem ~3.1]{BD00}} \label{thm:homeo_stabilizer}
	Let $\{u_1,\ldots,u_n\} \subset \frak{a}_{\Z}$ be as in Definition \ref{defn:ell_hypertoric} and set $\widetilde{\mu}_{\rA}= (\widetilde{\mu}_{\R,\rA},\widetilde{\mu}_{\C,\rA})$. Then
	\begin{itemize}
		\item The map $\widetilde{\mu}_{\rA}: \widetilde{\mathfrak{M}}^u_{\tau}(\alpha,\tilde{\beta}) \to \frak{a}^{\vee}_{\R}\oplus \frak{a}^{\vee}= \R^{3d}$ descends to a homeomorphism $\widetilde{\mathfrak{M}}^u_{\tau}(\alpha,\tilde{\beta})/\rA_{\R}\to \R^{3d}$. %In particular, this means $\widetilde{\mathfrak{M}}^u_{\tau}(\alpha,\beta)$ is connected.
		\item For $p\in \frak{a}^{\vee}_{\R}\oplus \frak{a}^{\vee}$, the stabilizer of a point in $\widetilde{\mu}_{\rA}^{-1}(p)$ is the subtorus in $\rA$ whose Lie algebra is spanned by the vectors $u_i$ for which $p\in \widetilde{\mathcal{A}}_{i,\gamma}$.
	\end{itemize}
\end{theorem}

The arrangement $\widetilde{\mathcal{A}}$ is called \textit{simple} if every subset of $m$-elements with nonempty intersection intersects in real codimension $3m$. The collection of vectors $u=\{u_1,\ldots,u_n\}$ is called \textit{unimodular} if every subset of $d$ linearly independent vectors in $u$ spans $\frak{a}_{\Z}$ over $\Z$.  

\begin{theorem} {\cite[Theorem ~3.2 \& 3.3]{BD00}}  \label{thm:smoothness}
	$\widetilde{\mathfrak{M}}^u_{\tau}(\alpha,\tilde{\beta})$ is an orbifold with at worst abelian quotient singularities if and only if every $\widetilde{\mathcal{A}}$ is simple. It is a smooth manifold if and only if, in addition, $u=\{u_1,\ldots,u_n\}$ is unimodular. 
\end{theorem}

%We say that $(\alpha,\tilde{\beta})$ is generic if every $d+1$ elements in $\widetilde{\mathcal{A}}$ have empty intersection. By the theorem above, this means $\widetilde{\mathfrak{M}}^u_{\tau}(\alpha,\tilde{\beta})$ is an orbifold. 

Theorem \ref{thm:homeo_stabilizer} and \ref{thm:smoothness} are counterparts of  Theorem ~3.1, 3.2 and 3.3 in \cite{BD00}. We will omit the proofs since they are almost verbatim. Together the theorems imply:
\begin{corollary}{\cite[Corollary ~3.5]{BD00}} \label{cor:fixed_points}
Suppose $\widetilde{\mathfrak{M}}^u_{\tau}(\alpha,\tilde{\beta})$ is an orbifold, then 
\begin{itemize}
\item The set of fixed points for the $\rA$-action is in one-to-one correspondence with the set of intersection points of $d$ elements in $\widetilde{\mathcal{A}}$.
\item If $p\in \frak{a}^{\vee}_{\R}\oplus \frak{a}^{\vee}$ lies in exactly $r$ elements in $\widetilde{\mathcal{A}}$, then the stabilizer of a point in $\widetilde{\mu}^{-1}_{\rA}(p)$ is an $r$-dimensional subtorus in $\rA$. 
\end{itemize}
\end{corollary}

\begin{theorem} \label{thm:topological_type}
Suppose $\widetilde{\mathfrak{M}}^u_{\tau}(\alpha,\tilde{\beta})$ and $\widetilde{\mathfrak{M}}^u_{\tau}(\alpha',\tilde{\beta}')$ are both generic, then
$\widetilde{\mathfrak{M}}^u_{\tau}(\alpha,\tilde{\beta})$ is homeomorphic to $\widetilde{\mathfrak{M}}^u_{\tau}(\alpha',\tilde{\beta}')$. 
\end{theorem}

Theorem \ref{thm:topological_type} is the counterpart of Theorem 6.1 in \cite{BD00}. Their argument involves hyperkähler rotation which is not available in our situation. Thus, we present an alternative proof here.

\begin{proof}
Let's set $\widetilde{\mu}_{\rK}=(\widetilde{\mu}_{\R,\rK},\widetilde{\mu}_{\C,\rK})$. We have $\widetilde{\mathfrak{M}}^u_{\tau}(\alpha,\tilde{\beta})=\widetilde{\mu}_{\rK}^{-1}(\alpha,\tilde{\beta})/\rK_{\R} $ and $\widetilde{\mathfrak{M}}^u_{\tau}(\alpha',\tilde{\beta}')=\widetilde{\mu}_{\rK}^{-1}(\alpha',\tilde{\beta}')/\rK_{\R}$. Therefore, it suffices to show that the level sets $\widetilde{\mu}_{\rK}^{-1}(\alpha,\tilde{\beta})$ and $\widetilde{\mu}_{\rK}^{-1}(\alpha',\tilde{\beta}')$ are $\rK_{\R}$-equivariantly diffeomorphic. The diffeomorphism can be constructed as follows.

Fisrt, we note that the critical values of $\widetilde{\mu}_{\rK}$ coincides with the set in $\frak{k}_{\R}^{\vee}\times \mathfrak{k}=\R^{3k}$ for which $d+1$ elements in $\widetilde{\mathcal{A}}$ have nonempty intersection, which is in real codimension three. Let $U$ denote the complement of this set. and note that we have $(\alpha,\tilde{\beta}), (\alpha',\tilde{\beta}')\in U$. Let $\ell':\R\to \R^{3k}$ be the straight line passing through $(\alpha,\tilde{\beta})$ and $(\alpha',\tilde{\beta}')$. 
If $\ell'$ is in $U$, then we set $\ell=\ell'$, otherwise, let $\ell$ be a small deformation of $\ell'$ in $U$ passing through $(\alpha,\tilde{\beta})$ and $(\alpha',\tilde{\beta}')$. Let $\cM_{\ell} \subset \widetilde{X}^n$ be the submanifold
\[
\cM_{\ell}=\{p\in \widetilde{X}^n| \widetilde{\mu}_{\rK}(p)\in \ell\}. 
\]
Let $F:\cM_{\ell}\to \R$ be the function defined by $F= P_{\vec{v}}\circ  t_{(\alpha,\tilde{\beta})}\circ \widetilde{\mu}_{\rK}|_{\cM_{\ell}}$ where $t_{(\alpha,\tilde{\beta})}:\R^{3k}\to \R^{3k}$ is the translation by $(\alpha,\tilde{\beta})$ and $P_{\vec{v}}: \R^{3k}\to \R$ is the projection to the line in the direction of $\vec{v}=(\alpha'-\alpha,\tilde{\beta}'-\tilde{\beta})$. By construction, $F$ has no critical points. 

Now, let $g$ be the restriction to $\cM_{\ell}$ of the standard metric on $\C^{3n}$, and let $\nabla F$ be the gradient vector field of $F$ with respect to $g$. We choose a point $p_0\in \cM_{\ell}$ and define a function $\rho:\cM_{\ell}\to [0,\infty)$ by the Riemannian distance to $p_0$ with respect to $g$. Since $g$ is a complete metric, the auxillary metric $g'$ defined by 
\[
g'=\frac{g}{(1+\rho)^2}
\]
is also complete (see \cite[Theorem ~1.2 \& 1.3]{Gli97}). By direct computation, we see that $\nabla F$ is bounded with respect to $g'$. This means $\nabla F$ is a complete vector field (see \cite[Theorem ~1.1]{Gli97}). Moreover, since both $g$ and $F$ are $\rK_{\R}$-invariant, $\nabla F$ is also $\rK_{\R}$-invariant. The flow of $\nabla F$ gives a $\rK_{\R}$-equivariant diffeomorphism between $\widetilde{\mu}_{\rK}^{-1}(\alpha,\tilde{\beta})$ and $\widetilde{\mu}_{\rK}^{-1}(\alpha',\tilde{\beta}')$.
\end{proof}

When $\tilde{\beta}$ is a lift of $\beta$, it is easy to see that $\widetilde{\mathfrak{M}}^u_{\tau}(\alpha,\tilde{\beta})$ is a covering space for $\mathfrak{M}^u_{\tau}(\alpha,\beta)$. The first part of Theorem \ref{thm:homeo_stabilizer} implies that  $\widetilde{\mathfrak{M}}^u_{\tau}(\alpha,\tilde{\beta})$ is connected. This also means $\mathfrak{M}^u_{\tau}(\alpha,\beta)$ is connected. The following proposition shows that $\widetilde{\mathfrak{M}}^u_{\tau}(\alpha,\tilde{\beta})$ is a universal cover for $\mathfrak{M}^u_{\tau}(\alpha,\beta)$.

\begin{prop}
$\pi_1(\widetilde{\mathfrak{M}}^u_{\tau}(\alpha,\tilde{\beta}))=0$. 
\end{prop}
\begin{proof}
Let's consider the map 
\[
\rho:\pi_1(\widetilde{\mathfrak{M}}^u_{\tau}(\alpha,\tilde{\beta}))\to \pi_1(\widetilde{\mathfrak{M}}^u_{\tau}(\alpha,\tilde{\beta})/\rA_{\R})
\] 
on fundamental groups induced by the quotient map. By Theorem \ref{thm:homeo_stabilizer}, the latter group is simply $\pi_1(\R^{3d})=0$. Let $\tilde{\gamma}$ be a loop in $\widetilde{\mathfrak{M}}^u_{\tau}(\alpha,\tilde{\beta})$ covering a contractible loop $\gamma$ in $\widetilde{\mathfrak{M}}^u_{\tau}(\alpha,\tilde{\beta})/\rA_{\R}$. Since $\widetilde{\mathfrak{M}}^u_{\tau}(\alpha,\tilde{\beta})\to \widetilde{\mathfrak{M}}^u_{\tau}(\alpha,\tilde{\beta})/\rA_{\R}$ is a $\rA_{\R}$-fibration, we may assume $\gamma$ is constant and $\tilde{\gamma}$ is contained in a $\rA_{\R}$-orbit. By Corollary \ref{cor:fixed_points}, the $\rA_{\R}$-action on $\widetilde{\mathfrak{M}}^u_{\tau}(\alpha,\tilde{\beta})$ has at least one fixed point and is 
is free on an open dense subset. We can therefore contract the loop $\tilde{\gamma}$ by deforming it to a fixed point through $\rA_{\R}$-orbits. This means $\rho$ is an isomorphism.
\end{proof}

Suppose $\tilde{\beta}$ is a lift of $\beta$, then the following properties for $\mathfrak{M}^u_{\tau}(\alpha,\beta)$ follows from the properties for its universal cover $\widetilde{\mathfrak{M}}^u_{\tau}(\alpha,\tilde{\beta})$ listed above.

\begin{theorem} \label{thm:homeo_stabilizer_ell}
	%Let $u_1,\ldots,u_n \in \frak{a}_{\Z}$ be as in Definition \ref{defn:ell_hypertoric} and 
	Let's set $\mu_{\rA}=(\mu_{\R,\rA},\mu_{\C,\rA})$. Then
	\begin{itemize}
		\item The map $\mu_{\rA}: \mathfrak{M}^u_{\tau}(\alpha,\beta) \to \frak{a}^{\vee}_{\R}\times E_{\tau}^d$ descends to a homeomorphism $\mathfrak{M}^u_{\tau}(\alpha,\beta)/\rA_{\R}\to \frak{a}^{\vee}_{\R}\times E_{\tau}^d$. %In particular, this means $\mathfrak{M}^u_{\tau}(\alpha,\beta)$ is connected.
		\item For $p\in \frak{a}^{\vee}_{\R}\times E_{\tau}^d$, the stabilizer of a point in $\mu^{-1}_{\rA}(p)$ is the subtorus in $\rA$ whose Lie algebra is spanned by the vectors $u_i$ for which $p\in \mathcal{A}_{i}$.
	\end{itemize}
\end{theorem}

The arrangement $\mathcal{A}$ is called \textit{simple} if every subset of $m$ elements with nonempty intersection intersects in real codimension $3m$.

\begin{theorem} \label{thm:smoothness_ell}
	$\mathfrak{M}^u_{\tau}(\alpha,\beta)$ is an orbifold with at worst abelian quotient singularities if and only if $\mathcal{A}$ is simple. It is a smooth manifold if and only if, in addition, $u=\{u_1,\ldots,u_n\}$ is unimodular
\end{theorem}

%We say that $(\alpha,\beta)$ is generic if every $d+1$ elements in $\mathcal{A}$ have empty intersection. By the theorem above, this means $\mathfrak{M}^u_{\tau}(\alpha,\beta)$ is an orbifold.

Theorem \ref{thm:homeo_stabilizer_ell} and \ref{thm:smoothness_ell} together imply
\begin{corollary}
	Suppose	$\mathfrak{M}^u_{\tau}(\alpha,\beta)$ is an orbifold, then 
	\begin{itemize}
		\item The set of fixed points for the $\rA$-action is in one-to-one correspondence with the set of intersection points of $d$ elements in $\mathcal{A}$.
		\item If $p\in \frak{a}^{\vee}_{\R}\oplus \frak{a}^{\vee}$ lies in exactly $r$ elements in $\mathcal{A}$, then the stabilizer of a point in $\mu^{-1}_{\rA}(p)$ is an $r$-dimensional subtorus in $\rA$. 
	\end{itemize}
\end{corollary}

\begin{theorem}%{\cite[Theorem ~6.1]{BD00}}
Suppose	$\mathfrak{M}^u_{\tau}(\alpha,\beta)$ and 	$\mathfrak{M}^u_{\tau}(\alpha',\beta')$ are both orbifolds, then $\mathfrak{M}^u_{\tau}(\alpha,\beta)$ is homeomorphic to $\mathfrak{M}^u_{\tau}(\alpha',\beta')$. 
\end{theorem}

\section{A comparison with the abelian Coulomb branches} \label{sec:BFN}
For a complex reductive group $G$ and $\bfM$ a symplectic $G$-representation, one can associate to the pair $(G,\bfM)$ two holomorphic symplectic varieties $\cM_{H}(G,\bfM)$ and $\cM_{C}(G,\bfM)$, which are the Higgs and Coulomb branches of the moduli space of vacua for the $3$-dimensional $\cN=4$ supersymmetric gauge theory. %In particular, when $G=\bK$ is abelian with the additive hypertoric variety $\frak{M}^+_u(0,0)$ (Definition \ref{defn:+_hypertoric}) is the Higgs branch $\cM_{H}(\rK,T^*\C^n)$. 
When $\bfM=T^*\bfN$ for a $G$-representation $\bfN$, the Coulomb branch was defined in \cite{BFN16} as an affine variety whose coordinate ring is the equivariant Borel-Moore homology of the variety of triples over the affine Grassmanian of $G$, equipped with a convolution type product. 

When $G=\rK$ is abelian $\bfN\cong \C^n$ is the restriction of the standard representation of $\rT\cong (\C^{\times})^n$ to its subtorus $\rK$ defined by the short exact sequence \eqref{eq:ses3}, the Higgs branch is a hypertoric variety, i.e., we have
\[
\cM_{H}(\rK,T^*\C^n)=\frak{M}^+_u(0,0),
\]
where $\frak{M}^+_u(0,0)$ is defined as in Definition \ref{defn:+_hypertoric}. On the other hand, the Coulomb branch $\cM_{C}(\rK,T^*\C^n)$ is also an additive hypertoric variety. In this section, we will compare the ``coordinate rings'' of elliptic hypertoric varieties with that of additive and multiplicative hypertoric varieties arising from the BFN Coulomb branch construction. 

Let's consider the short exact sequences
\begin{equation}  \label{eq:ses4}
0\longrightarrow \mathfrak{a}^{\vee} \overset{\pi^{\vee}}{\longrightarrow} \mathfrak{t}^{\vee} \overset{\iota^{\vee}}{\longrightarrow} \mathfrak{k}^{\vee} \longrightarrow 0,
\end{equation} 
\begin{equation} \label{eq:ses5}
0\longleftarrow \mathfrak{a}\overset{\pi}{\longleftarrow}\mathfrak{t}\overset{\iota}{\longleftarrow}\mathfrak{k}\longleftarrow 0,
\end{equation}
dual to the sequences \eqref{eq:ses1} and \eqref{eq:ses2}. Exponentiating \eqref{eq:ses4} gives us a short exact sequence of tori dual to \eqref{eq:ses3}
\begin{equation} \label{eq:ses6}
1\longrightarrow \rA^{\vee} \overset{\psi^{\vee}}{\longrightarrow} \rT^{\vee} \overset{\phi^{\vee}}{\longrightarrow} \rK^{\vee}\longrightarrow 1.
\end{equation}
The abelian Coulomb branch $\cM_{C}(\rA^{\vee},T^*\C^n)$ is the same hypertoric variety
\[
\cM_{C}(\rA^{\vee},T^*\C^n)= \mathfrak{M}_u^+(0,0).
\]
%The additive hypertoric variety $\mathfrak{M}_u^+(0,0)$  coincides with $\cM_{C}(\rA^{\vee},T^*\C^n)$.
The coordinate ring of $\cM_{C}(\rA^{\vee},T^*\C^n)$ has the following explicit presentation (cf. 5(ii) in \cite{Nak15} and Theorem 4.1 in \cite{BFN16}):  

For two integers $\ell, m$, let's set
\begin{equation} \label{eq:delta}
\delta(\ell,m)=\begin{cases*}
0 & \text{if $\ell$ and $m$ have the same sign,}\\
\min(|\ell|,|m|) & \text{if $\ell$ and $m$ have different signs.}
\end{cases*}
\end{equation}
Let $\rho_+:\C[\frak{t}^{\vee}]\to \C[\frak{a}^{\vee}]$ be the map induced by $\pi^{\vee}$, and let
$\zeta_1,\ldots,\zeta_n$ be the generators of $\C[\frak{t}^{\vee}]$. We have 
\begin{equation} \label{eq:+_ring}
\C[\cM_{C}(\rA^{\vee},T^*\C^n)]=\bigoplus_{\lambda\in \frak{a}^{\vee}_{\Z}} \C[\frak{a^{\vee}}] r_+^{\lambda},
\end{equation}
as a vector space, and the multiplication is given by 
\begin{equation} \label{eq:+_relations}
r_+^{\lambda}r_+^{\mu}=r_+^{\lambda+\mu} \rho_+\left( \prod_i^n \zeta_i^{\delta(\lambda_i,\mu_i)}\right),
\end{equation}
where $\lambda_i=\langle \lambda,e_i\rangle$, $\mu_i=\langle \mu,e_i\rangle$, and $\langle -,-\rangle$ is the natural pairing between $\frak{t}$ and $\frak{t}^{\vee}$.  
We note that $\psi^{\vee}$ induces a map between the classifying spaces $B\rA^{\vee}\to B\rT^{\vee}$. Thus, $\pi_+$ can naturally be viewed as the pullback map on cohomology. %$H^*_{T^{\vee}}(\pt)\to H^*_{D^{\vee}}(\pt)$. 
From this viewpoint, $\zeta_i$ are the universal first Chern class in $H^*_{\rT^{\vee}}(\pt)=H^*(B\rT^{\vee})$, i.e., $\zeta_i$ is the first Chern class of a universal line bundle over $B\rT^{\vee}=(\C\bP^{\infty})^n$.

%=H^*((\C\bP^{\infty})^n)$. %tautological line bundle over $BT^{\vee}=(\C\bP^{\infty})^n$.  
%Note that we were to view $\pi$ as the map on cohomology $H^*(BD^{\vee})\to H^*(BT^{\vee})$ induced by the map of classifying spaces $BD^{\vee}\to BT^{\vee}$

In \cite{BFN16}, the authors also defined the K-theoretic Coulomb branch associated to a pair $(G,T^*\bfN)$ (see also \cite{Tel21}). %whose coordinate ring is defined using $\rk$-group instead of homology. 
The multiplicative hypertoric variety $\mathfrak{M}_u^{\times}(0,0)$ defined using quasi-Hamiltonian reduction (Definition \ref{defn:*_hypertoric}) coincides with the K-thereotic Coulomb branch associated to $(\rA^{\vee},T^*\C^n)$,  whose coordinate ring is generated by the ring $\C[\rA^{\vee}]$ together with the symbols $r^{\lambda}_{\times}$, 
%\[
%\bigoplus_{\lambda\in \frak{a}_{\Z}^{\vee}} \C[D] r^{\lambda}, 
%\]
\begin{equation} \label{eq:*_ring}
\C[\mathfrak{M}_u^{\times}(0,0)]=\bigoplus_{\lambda\in \frak{a}^{\vee}_{\Z}} \C[\rA^{\vee}] r_{\times}^{\lambda},
\end{equation}
subject to the relations
\begin{equation} \label{eq:*_relations}
r^{\lambda}_{\times}r^{\mu}_{\times}=r^{\lambda+\mu}_{\times} \rho_{\times} \left(\prod_i^n (1-t_i)^{\delta(\lambda_i,\mu_i)}\right).
\end{equation}
Here $\rho_{\times}:\C[\rT^{\vee}] \to \C[\rA^{\vee}]$ is the map induced by $\psi^{\vee}$, and $t_1,\ldots,t_n$ are generators of $\C[\rT^{\vee}]$. We note that $t_i$ is the class of a universal line bundle over $B\rT^{\vee}$ and $(1-t_i)$  is its first Chern class in $K_{\rT^{\vee}}(\pt)=K(B\rT^{\vee})$. %$t_i$ can be viewed as the class of a tautological line bundle over $BT^{\vee}$, and the term $(1-t_i)$ is its first Chern class in $\rk$-theory.

From the above presentations of coordinate rings of additive and multiplicative hypertoric varieties, the most natural generalization to an elliptic analogue would be to replace $\C[\frak{a}^{\vee}]$ and $\C[\rA^{\vee}]$ with the equivariant elliptic cohomology $\mathrm{Ell}_{\rA^{\vee}}(\pt)$ (see Section \ref{sec:Hikita} for the definition), and to replace the pullbacks of $\zeta_i$ and $(1-t_i)$ by classes in $\mathrm{Ell}_{\rA^{\vee}}(\pt)$. However, this does not immediately work, since $\mathrm{Ell}_{\rA^{\vee}}(\pt)$ is defined to be the abelian variety $E_{\tau}^d$. %(The equivariant elliptic cohomology $\mathrm{Ell}_{G}(-)$ is a covariant functor that associates to every $G$-variety $X$ a scheme $\mathrm{Ell}_{G}(X)$ over $\mathrm{Ell}_{G}(\pt)$.) 
To resolve this, we can try to interpret $\mathrm{Ell}_{\rA^{\vee}}(\pt)$ as a graded ring of sections of some line bundles over $E_{\tau}^d$. As we will see momentarily, for suitable choice of line bundles, this does lead to our definition of elliptic hypertoric varieties.

Let $\psi^{\vee}_{\tau}:E_{\tau}^d \to E_{\tau}^n$ be the embedding which fits into the short exact sequence 
\[
1\longrightarrow E_{\tau}^d\overset{\psi^{\vee}_{\tau}}{\longrightarrow} E_{\tau}^n \overset{\phi^{\vee}_{\tau}}{\longrightarrow} E_{\tau}^k \longrightarrow 1.
\]
For $i=1,\ldots,n$, let $\bar{\cL}_i$ be the line bundle over $E_{\tau}^d$ which is the pullback of $\cO(e)$ %over $E_{\tau}$ 
via the map
\[
E_{\tau}^d\overset{\psi^{\vee}_{\tau}}{\longrightarrow} E_{\tau}^n \overset{\mathrm{pr}_i}{\longrightarrow} E_{\tau},
\]
where $\mathrm{pr}_i$ is the projection to the $i$-th component. We denote by $\cL_i$ the pullback of $\bar{\cL}_i$ to $\mathfrak{M}^{\tau}_u(0,0)$ via the moment map $\mu_{\C,\rA}$. We note that the hyperplane $H_{\tau, i}$ in $E_{\tau}^d$ defined by \eqref{eq:ell_hyperplane} is a Cartier divisor for $\cL_i$. %and the ample line bundle $\cL=\bigotimes_{i=1}^n \cL_i$ corresponds to the projective embedding for $\mathfrak{M}^{\tau}_u(0,0)$ in Definition \ref{defn:ell_hypertoric}. 

In order the match with the coordinate rings above, we will consider the $\Z^n$-graded rings $\mathrm{R}(E_{\tau}^d)$ and $\mathrm{R}(\mathfrak{M}^{\tau}_u)$ defined by
\[
\mathrm{R}(E_{\tau}^d) = \bigoplus_{\ell_1,\ldots,\ell_n \ge 0} H^0(E_{\tau}^d, \bar{\cL}_1^{\ell_1} \otimes \ldots \otimes \bar{\cL}_n^{\ell_n}),
\]
and
\[
\mathrm{R}(\mathfrak{M}^{\tau}_u) = \bigoplus_{\ell_1,\ldots,\ell_n \ge 0} H^0(\mathfrak{M}^{\tau}_u, \cL_1^{\ell_1} \otimes \ldots \otimes \cL_n^{\ell_n}).
\]
The the multi-Proj of $\mathrm{R}(E_{\tau}^d)$ gives a toric embedding $E_{\tau}^d \into \bP(2,3,1)^n$ which factors through through $\psi^{\vee}_{\tau}$.\footnote{The line bundle $\cO(e)$ over $E_{\tau}$ gives an embedding of $E_{\tau}$ into the weighted projective space $\bP(2,3,1)$ whose image is a degree $6$ curve of the form
	\[
	Y^2+a_1XYZ+A_3YZ^3=X^3+a_2X^2Y^2+a_2XZ^4+a_6Z^6.
	\]
	%This differs from the usual projective embedding of $E_{\tau}$ into $\bP^2$ given by $\cO(3e)$ whose image is a Weierstrass cubic curve.
}
$\mathrm{R}(\mathfrak{M}^{\tau}_u)$ is an algebra over $\mathrm{R}(E_{\tau}^d)$ whose generators and relations can be described as follows.

The map $\psi^{\vee}_{\tau}:E_{\tau}^d \to E_{\tau}^n$  has the form
\[
\psi^{\vee}_{\tau}(y_1,\ldots,y_d)=\left(\sum_{j=1}^d \psi_{1j}y_j, \ldots,\sum_{j=1}^d \psi_{nj}y_j  \right)  
\] 
where $(\psi_{ij})\in \mathrm{Mat}_{n\times d}(\Z)$. For $i=1,\ldots,n$, let $\bar{\vartheta}_i:\C^d\to \C$ be the theta function
\[
\bar{\vartheta}_i(y_1,\ldots,y_d)=\vartheta\left(\sum_{j=1}^d \psi_{ij}y_j \right).
\]
%$\vartheta_i$ can be viewed as the pullback to $BD^{\vee}$ of the first Chern class of a tautological line bundles in 
It is easy to see that $\bar{\vartheta}_i\in \mathrm{R}(E_{\tau}^d)$. If we interpret the elliptic classes of $B\rA^{\vee}$ as elements in $\mathrm{R}(E_{\tau}^d)$, then $\bar{\vartheta}_i$ are the pullback of the universal first Chern classes in $\mathrm{Ell}(B\rT^{\vee})$. 

The generators of $\mathrm{R}(\mathfrak{M}^{\tau}_u)$ over $\mathrm{R}(E_{\tau}^d)$ are $\rK$-invariant monomials in variables $z_i$ and $w_i$ for $i=1,\ldots,n$, which are of the form
\begin{equation} \label{eq:invariant_monomial}
\prod_i^n \left(z_i^{\max(\lambda_i,0)}  w_i^{\max(-\lambda_i,0)} \right), 
\end{equation}
for $\lambda\in \frak{a}_{\Z}^{\vee}\subset \frak{t}_{\Z}^{\vee} $. Let's denote the monomial \eqref{eq:invariant_monomial} by the symbol $r^{\lambda}$. Then, we have 
\[
\mathrm{R}(\mathfrak{M}^{\tau}_u) =\bigoplus_{\lambda\in \frak{a}^{\vee}_{\Z}} \mathrm{R}(E_{\tau}^d) r^{\lambda}
\]
and the multiplication is given by
\begin{equation*} 
r^{\lambda}r^{\mu}=r^{\lambda+\mu}\prod_i^n \bar{\vartheta}_i^{\delta(\lambda_i,\mu_i)}.
\end{equation*} 
The $\Z^n$-grading of the symbols $r^{\lambda}$ are determined by $\deg z_i=0$ and $\deg w_i=(\delta_{i1},\ldots,\delta_{in})$. By construction, 

\begin{theorem}
We have 
\[
\mathfrak{M}^{\tau}_u(0,0)=\Spec_{E_{\tau}^d} \widetilde{\mathrm{R}(\mathfrak{M}^{\tau}_u)},
\]
where the RHS denotes the relative spectrum over $E_{\tau}^d$ of the sheaf  $\widetilde{\mathrm{R}(\mathfrak{M}^{\tau}_u)}$ of $\cO_{E_{\tau}^d}$-algebras associated to $\mathrm{R}(\mathfrak{M}^{\tau}_u)$.
\end{theorem}

%the relative spectrum of $\mathrm{R}(\mathfrak{M}^{\tau}_u(0,0))$ over $\mathrm{R}(E_{\tau}^d)$ is the elliptic hypertoric variety $\mathfrak{M}^{\tau}_u(0,0)$.

\section{The Hikita Conjecture} \label{sec:Hikita}
The additive hypertoric varieties $T^*\C^n\threebar \rK$ and $T^*\C^n\threebar \rA^{\vee}$ are dual symplectic singularities in the sense of symplectic duality \cite{BLPW14,BDGH16}. In \cite{Hik17}, Hikita conjectured if $X\to Y$ and $X^{!}\to Y^{!}$ are dual pairs of symplectic resolutions, then there is an isomorphism $H^*(X)\cong \C[(Y^!)^{\rA}]$, where $\rA$ denotes the maximal torus in the Hamiltonian automorphism group of $X$. This was proved for additive hypertoric varieties in the same paper. 
For additive hypertoric varieties and Springer resolutions, the equivariant and quantum extension of Hikita's conjecture was proved in \cite{KMP21} relating the quantized coordinate ring of $Y^!$ with the quantum cohomology of $X$. On the other hand, a $K$-theoretic analogue of \cite{KMP21} was proved in \cite{Zho21}, which relates the quantum $K$-theory of the Higgs branch of a $3d$ $\mathcal{N}=4$ supersymmetric gauge theory with the quantized coordinate ring of its $K$-theoretic Coulomb branch.

In this section, we establish an isomorphism between the $\rK^{\vee}$-equivariant elliptic cohomology %$\mathrm{Ell}_{\rK^{\vee}}(\mathfrak{M}^+_v(\alpha,\beta))$  
of a smooth additive hypertoric variety %$\mathfrak{M}^+_v(\alpha,\beta)=
$T^*\C^n\threebar_{(\alpha,\beta)} \rA^{\vee}$ 
and the $\rA$-fixed point variety in $X_{\vartheta}^n\twobar \rK$. The later is an $\rA$-equivariant deformation space (of holomorphic Poisson structures) for the elliptic hypertoric variety $\mathfrak{M}^{\tau}_u(0,0)$.

\subsection{Equivariant elliptic cohomology of additive hypertoric vareties}
Let $v=\{v_1,\ldots,v_n\}$ be a collection of primitive vectors in $\frak{k}_{\Z}^{\vee}$. We will assume $v$ is unimodular. Let $\rA^{\vee}$ be the subtorus in $\rT^{\vee}$ determined by $v$ as in \eqref{eq:ses6}. For $(\alpha,\beta)\in \frak{a}_{\R}\oplus \frak{a}$, we denote by $\mathfrak{M}^+_v(\alpha,\beta)$ the additive hypertoric variety
\[
\mathfrak{M}^+_v(\alpha,\beta)= T^*\C^n\threebar_{(\alpha,\beta)} \rA^{\vee}.
\]
We will assume $\mathfrak{M}^+_v(\alpha,\beta)$ to be smooth. Since the topology of $\mathfrak{M}^+_v(\alpha,\beta)$ is independent of $(\alpha,\beta)$ whenever it is smooth (see \cite[Theorem ~6.1]{BD00}), we can set $(\alpha,\beta)=(\alpha,0)$ for generic $\alpha$.

%In this subsection, we give a description for the $\rK^{\vee}$-equivariant elliptic cohomology of a smooth additive hypertoric variety %$\mathfrak{M}^+_v(\alpha,\beta)$.
We are interested in a description of the $\rK^{\vee}$-equivariant elliptic cohomology of $\mathfrak{M}^+_v(\alpha,\beta)$, for the residual $\rK^{\vee}$-action. For this purpose, we follow the the treatment of equivariant elliptic cohomology in \cite{AO16, RSVZ19} as it is closest to our setting. 
We will refer to \cite{Gro94,GKV95,Lur09} for more detailed discussions on the subject.

For a torus $\rT=(\C^{\times})^n$, the equivariant elliptic cohomology is a covariant functor 
\[
\mathrm{Ell}_{\rT}: \{\rT\text{ - varieties}\} \longrightarrow \{ \text{schemes}\},
\]
which assigns to a $\rT$-variety $X$ a scheme $\mathrm{Ell}_{\rT}(X)$. In particular, it assigns to a point the abelian variety $E_{\tau}^n$. 
\[
\mathrm{Ell}_{\rT}(\pt)=\rT/q^{\frak{t}_{\Z}}= E_{\tau}^n.  %\qquad n=\dim \rT.
\]
To the canonical projection $X\to \pt$, it associates the map 
\[
\pi:\mathrm{Ell}_{\rT}(X)\to E_{\tau}^n.
\] 
%whose fibers can be described as follows. 
%The scheme structure of $\mathrm{Ell}_{\rT}(X)$ can be described in term of the map $\pi$. 
For each $x\in E_{\tau}^n$, we take a small analytic neighborhood $\mathrm{U}_x$ isomorphic via the exponential map 
\[
%\exp: 
\frak{t}= \mathrm{Lie}(E_{\tau})\otimes _{\Z} \frak{t}_{\Z} \longrightarrow E_{\tau}\otimes_{\Z} \frak{t}_{\Z}=E_{\tau}^n
\]
to a small analytic neighborhood in $\frak{t}\cong \C^n$. Let $\cO_{E_{\tau}^n}$ denote the sheaf of holomorphic functions on $E_{\tau}^n$ and let $\mathscr{H}_{\rT}$ be the sheaf of $\cO_{E_{\tau}^n}$-algebras 
\begin{equation*}
\mathscr{H}_{\rT}(\mathrm{U}_x)=H^*_{\rT}(X^{\rT_x})\otimes_{H^*_{\rT}(\pt)} \cO_{\mathrm{U}_x}, %^{\mathrm{an}}
\end{equation*}
where $\rT_x\subset \rT$ is the subgroup
\begin{equation} \label{eq:T_x}
\rT_x= \bigcap_{\chi(x)=q^{\Z}} \ker \chi, 
\end{equation}
and $\chi$ ranges over all characters
\[
\chi\in \frak{t}_{\Z}\cong \Hom(E_{\tau}^n,E_{\tau}). 
\]
%This means $\rT_x\subset \rT$ is the smallest subgroups such that $\mathrm{Ell}_{\rT_x}(\pt)\subset E_{\tau}^n$ contains $x$. 
The scheme $\mathrm{Ell}_{\rT}(X)$ is defined to be the relative spectrum \cite{Gro94}
\begin{equation*}
\mathrm{Ell}_{\rT}(X):=\Spec_{E_{\tau}^n} \mathscr{H}_{\rT}. 
\end{equation*}
The fibers of $\pi$ can be described in the following diagram \cite{AO16}:
\begin{equation}
\label{eq:pi}
\begin{tikzcd}
%\Spec H^*(X^{\rT_x};\C)  \arrow[hookrightarrow]{r} \arrow{d}  & \Spec H^*_{\rT}(X^{\rT_x};\C)  \arrow{d}   &  \pi^{-1}(\mathrm{U}_x)
%\arrow{l} \arrow{r} \arrow{d} & \mathrm{Ell}_{\rT}(X)  \arrow{d}  \\
%\{x\}  \arrow[hookrightarrow]{r} & \C^n    &  \mathrm{U}_x \arrow{l} \arrow{r}   & E_{\tau}^n . 
\Spec H^*_{\rT}(X^{\rT_x};\C)  \arrow{d}   &  \pi^{-1}(\mathrm{U}_x)
 \arrow{l} \arrow{r} \arrow{d} & \mathrm{Ell}_{\rT}(X)  \arrow{d}  \\
\C^n    &  \mathrm{U}_x \arrow{l} \arrow{r}   & E_{\tau}^n . 
\end{tikzcd}
\end{equation}

In general, it is difficult to describe the scheme structure of $\mathrm{Ell}_{\rT}(X)$. However, when $X$ is a GKM variety, $\mathrm{Ell}_{\rT}(X)$ has a nice combinatorial description. %in terms of the $0$ and $1$-dimensional orbits of the $\rT$-action.

\begin{defn} %{\cite[Definition ~1]{RSVZ19}}
A smooth $\rT$-variety $X$ is called a GKM variety if
\begin{enumerate}
%\item For any subgroup $\rT'\subset \rT$, $H^*_{\rT}(X^{\rT'})$ is equivariantly formal (see \cite{GKM98} for more details on equivariant formality), i.e, it is a free module over $H^*_{\rT}(\pt)$.    

\item  $X$ is equivariantly formal in the sense of \cite{GKM98}, i.e, $H^*_{\rT}(X)$ is a free module over $H^*_{\rT}(\pt)$. %(See \cite{GKM98} for more details on equivariant formality.)

\item The $\rT$-action on $X$ has finitely many fixed points and $1$-dimensional orbits.

%\item The equivariant $1$-skeleton $X_1$ 
%\[
%X_1=\{x\in X| \mathrm{codim}( \mathrm{stab}(x)) = 1  \}
%\]  
%is complex $1$-dimensional.
%\item For every pair of fixed points $p,q\in X^{\rT}$ there is at most one $\rT$-equivariant rational curve connecting them.
%The set of fixed points $X^{\rT}$  consists of finitely many isolated points.
%\item For every pair of fixed points $p,q\in X^{\rT}$ there is at most one $\rT$-equivariant rational curve connecting them.
\end{enumerate}
\end{defn}

%We note that equivariant formality implies 
%\begin{equation*}
%H^*_{\rT}(X^{\rT'})=H^*(X^{\rT'})\otimes H^*_{\rT}(\pt).
%\end{equation*}
%This implies that the fiber of $\pi$ over $x\in E_{\tau}^n$ is given by
%\[
%\pi^{-1}(x)=\Spec H^*(X^{\rT_x};\C).
%\]

It is shown in \cite{HH04} that $\mathfrak{M}^+_v(\alpha,\beta)$ is a GKM variety for the $\bm{\rK}:=\C^{\times}_{\hbar}\times \rK^{\vee}$-action, where $\C^{\times}_{\hbar}:=\C^{\times}$ acts on $\mathfrak{M}^+_v(\alpha,\beta)$ by %the conical action 
\[
t\cdot[\vec{z},\vec{w}]=[\vec{z},t\vec{w}], \quad [\vec{z},\vec{w}]\in \mathfrak{M}^+_v(\alpha,\beta), \quad t\in \C^{\times}. 
\]
%(See \cite{HP04} for properties of this action.) 
%Condition (1) follows from \cite[Lemma ~2.1]{AO16} and (2) and (3) follow from \cite{HH04}. 
However, the $\rK^{\vee}$-action on $\mathfrak{M}^+_v(\alpha,\beta)$ is not GKM.   For example, it has infinitely many $1$-dimensional $\rK^{\vee}$-orbits. Nonetheless, we will find  $\mathrm{Ell}_{\rK^{\vee}}(\mathfrak{M}^+_v(\alpha,\beta))$ to be a closed subscheme of $\mathrm{Ell}_{\bm{\rK}}(\mathfrak{M}^+_v(\alpha,\beta))$. 
By \cite[Propostion ~1]{RSVZ19}, the latter has a relatively simple description in term of the $0$ and $1$-dimensinal $\bm{\rK}$-orbits.  %in $\mathfrak{M}^+_v(\alpha,\beta)$. %similar to the GKM description of equivariant cohomology.

Now, let's give an explicit description for $\mathrm{Ell}_{\bm{\rK}}(\mathfrak{M}^+_v(\alpha,\beta))$. Let $\hbar,x_1,\ldots,x_n\in \C$ be the additive coordinates on $E_{\tau,\hbar}\times E_{\tau}^{n}$ where the component $E_{\tau,\hbar}:=E_{\tau}$ has the coordinate $\hbar$. Let $\cO(e)_0,\ldots,\cO(e)_n$ be the line bundles over $E_{\tau,\hbar}\times E_{\tau}^{n}$ defined by the theta functions $\vartheta(\hbar)$ and $\vartheta(x_1),\ldots,\vartheta(x_n)$, respectively. Let $\mathrm{R}(E_{\tau,\hbar}\times E_{\tau}^{n})$ the $\Z^{n+1}$-graded ring
\[
\mathrm{R}(E_{\tau,\hbar}\times E_{\tau}^{n})=\bigoplus_{\ell_0,\ldots,\ell_n \ge 0} H^0(E_{\tau,\hbar} \times E_{\tau}^{n}, \cO(e)_0^{\ell_0} \otimes \ldots \otimes \cO(e)_n^{\ell_n}), 
\]

We define a \textit{circuit} in $v$ to be a minimal subset $S\subset\{1,\ldots,n\}$ such that the vectors $v_i$ for $i\in S$ are linearly dependent. Since $\mathfrak{M}^+_v(\alpha,\beta)$ is smooth, we can fix a splitting $S=S^+\coprod S^-$ for each circuit $S$ such that if
\[
\beta_S:=\sum_{i\in S^+} e_i^{\vee}-\sum_{i\in S^-} e_i^{\vee} \in \ker(\frak{t}^{\vee}\to \frak{k}^{\vee}),
\]
then $\langle \beta_S, \hat{\alpha}\rangle \ge 0$, for any lift $\hat{\alpha}\in \frak{t}^{\vee}$ of $\alpha$.

%For a circuit $S\in \mathrm{Circ}(v)$, we put 
%\[
%\vartheta_S:= \prod_{i\in S^+} \vartheta(x_i) \prod_{i\in S^-} \vartheta(\hbar-x_i).
%\]
\begin{theorem} \label{thm:ell_hbar_M}	
We have
\begin{equation} 
\mathrm{Ell}_{\bm{\rK}}(\mathfrak{M}^+_v(\alpha,\beta)) =  \Proj  \mathrm{R}(E_{\tau,\hbar} \times E_{\tau}^{n})/\left(\vartheta_S\cdot r\in  \mathrm{R}(E_{\tau,\hbar} \times E_{\tau}^{n})   | r\in \mathrm{R}(E_{\tau,\hbar} \times E_{\tau}^{n}), S\in \mathrm{Circ}(v) \right) ,
\end{equation}
where 
\[
\vartheta_S:= \prod_{i\in S^+} \vartheta(x_i) \prod_{i\in S^-} \vartheta(\hbar-x_i),
\] 
and ``Proj'' denotes the multi-Proj of a $\Z^{n+1}$-graded ring.	
\end{theorem}

\begin{proof}
By \cite[Theorem ~3.5]{HH04}, if we glue copies $\C^{n+1}$ according to the $0$ and $1$-dimensinal $\bm{\rK}$-orbits in $\mathfrak{M}^+_v(\alpha,\beta)$, %GKM data, 
the resulting variety is 
\[
\Spec H^*_{\bm{\rK}}(\mathfrak{M}^+_v(\alpha,\beta)  ;\C)   =\Spec \left( \C[\hbar,x_1,\ldots,x_n]/\left(\prod_{i\in S^+} x_i \prod_{i\in S^-} (\hbar-x_i)| S\in \mathrm{Circ}(v) \right) \right).
\]
Then by \cite[Propostion ~1]{RSVZ19} (see also Appendix A in \cite{Ros03}), the same gluing gives $\mathrm{Ell}_{\bm{\rK}}(\mathfrak{M}^+_v(\alpha,\beta))$ if we replace $\C^{n+1}$ with $E_{\tau}^{n+1}$. 
\end{proof}

\begin{theorem} \label{thm:ell_X}
	$\mathrm{Ell}_{\rK^{\vee}}(\mathfrak{M}^+_v(\alpha,\beta))$ is the fiber product 
	%$\mathrm{Ell}_{\rA^{\vee}}(\mathfrak{M}^+_v(\alpha,\beta))$ is a subvariety of $\mathrm{Ell}_{\rA^{\vee}\times \C^{\times}_{\hbar}}(\mathfrak{M}^+_v(\alpha,\beta))$ defined by $\vartheta_{0}=0$. In other words, we have
	%\begin{equation} 
	%\mathrm{Ell}_{\rK^{\vee}}(\mathfrak{M}^+_v(\alpha,\beta)) = \Proj \left( \mathrm{R}(E_{\tau}^n)/\left(\prod_{i\in S} \vartheta(x_i) | S\in \mathrm{Circ}(v) \right) \right).
	%\end{equation}
	%Here we have used ``Proj'' to denote the multi-Proj of a multi-graded ring.	
	\begin{equation*} 
	\begin{tikzcd}
	\mathrm{Ell}_{\bm{\rK}}(\mathfrak{M}^+_v(\alpha,\beta))\times_{E_{\tau,\hbar} \times E_{\tau}^{k}}     E_{\tau}^k \arrow{d}   \arrow{r}  &  \mathrm{Ell}_{\bm{\rK}}(\mathfrak{M}^+_v(\alpha,\beta))
	\arrow{d}  \\
	E_{\tau}^k   \arrow{r}  &  E_{\tau,\hbar} \times E_{\tau}^{k}
	\end{tikzcd}
	\end{equation*}
	in the category of schemes, where $E_{\tau}^k\to E_{\tau,\hbar}\times E_{\tau}^{k}$ is the inclusion to the second component and $\mathrm{Ell}_{\bm{\rK}}(\mathfrak{M}^+_v(\alpha,\beta))\to E_{\tau,\hbar}\times E_{\tau}^{k}$ is the map associated to $\mathfrak{M}^+_v(\alpha,\beta)\to \pt$. %by $\mathrm{Ell}_{\bm{\rK}}$. 
\end{theorem}

For the proof of Theorem \ref{thm:ell_X}, we need the following lemmas:
\begin{lemma} \label{lem:formal_bk}
	For any subgroup $\bm{\rK}'\subset \bm{\rK}$, we have 
	\[
	H^*_{\bm{\rK}}(\mathfrak{M}^+_v(\alpha,\beta)^{\bm{\rK}'})=H^*(\mathfrak{M}^+_v(\alpha,\beta)^{\bm{\rK}'})\otimes 	H^*_{\bm{\rK}}(\pt). 
	\]
\end{lemma}

The proof for Lemma \ref{lem:formal_bk} is similar to that of \cite[Lemma ~2.1]{AO16} (which proved the same statement for Nakajima's quiver varieties), 
and is hence omitted.
%\begin{proof}
%The proof is similar to that \cite[Lemma ~2.1]{AO16} for Nakajima's quiver varieties. We choose a generic one parameter subgroup $\C^{\times} \subset \bm{\rK}$ such that the limit 
%\[
%\lim_{s\to 0} s\cdot p\in \mathfrak{M}^+_v(\alpha,\beta)^{\bm{\rK}}, \quad s\in \C^{\times}, 
%\]
%exists for all $p\in \mathfrak{M}^+_v(\alpha,\beta)$. For any subgroup $\bm{\rK}'\subset \bm{\rK}$, we have the corresponding Białynicki-Birula decomposition
%\[
%\mathfrak{M}^+_v(\alpha,\beta)^{\bm{\rK}'}= \bigcup_{p_i\in \mathfrak{M}^+_v(\alpha,\beta)^{\bm{\rK}}} \{p\in \mathfrak{M}^+_v(\alpha,\beta)^{\bm{\rK}'}| \lim_{s\to 0} s\cdot p =p_i \}.
%\]
%Induction on the partial order among strata and the long exact sequence of a pair show that $H_*(\mathfrak{M}^+_v(\alpha,\beta)^{\bm{\rK}'})$ is generated by $\bm{\rK}$-invariant cycles. It then follows from \cite[Corollary ~1.3.2]{GKM98} that $\mathfrak{M}^+_v(\alpha,\beta)^{\bm{\rK}'}$ is equivariantly formal.
%\end{proof}

\begin{lemma} \label{lem:formal_k}
	For any subgroup $\rK'\subset \rK^{\vee}$, we have 
	\[
	H^*_{\rK^{\vee}}(\mathfrak{M}^+_v(\alpha,\beta)^{\rK'})=H^*(\mathfrak{M}^+_v(\alpha,\beta)^{\rK'})\otimes 	H^*_{\rK^{\vee}}(\pt). 
	\]
\end{lemma}
\begin{proof}
Since any subgroup $\rK'\subset \rK^{\vee}$ is automatically a subgroup of $\bm{\rK}$ and a $\bm{\rK}$-invariant cycle is also $\rK^{\vee}$-invariant, it follows from the proof of \cite[Lemma ~2.1]{AO16} that $H_*(\mathfrak{M}^+_v(\alpha,\beta)^{\rK'})$ is generated by $\rK^{\vee}$-invariant cycles. By \cite[Corollary ~1.3.2]{GKM98}, $\mathfrak{M}^+_v(\alpha,\beta)^{\rK'}$ is equivariantly formal. 
\end{proof}

\begin{proof}[Proof of Theorem \ref{thm:ell_X}]
	Let $\mathscr{H}_{\rK^{\vee}}$ and $\mathscr{H}_{\bm{\rK}}$ be the sheaves of algebras over $E_{\tau}^k$ and $E_{\tau,\hbar}\times E_{\tau}^{k}$ whose relative spectrum define $\mathrm{Ell}_{\rK^{\vee}}(\mathfrak{M}^+_v(\alpha,\beta))$ and $\mathrm{Ell}_{\bm{\rK}}(\mathfrak{M}^+_v(\alpha,\beta))$, respectively. For a point $a\in E_{\tau}^k\subset E_{\tau,\hbar}\times E_{\tau}^{k}$, 
	we have
	\begin{align*}
	(\mathscr{H}_{\bm{\rK}})_a&=H^*_{\bm{\rK}}(\mathfrak{M}^+_v(\alpha,\beta)^{\bm{\rK}_a})\otimes _{H^*_{\bm{\rK}}(\pt)} \cO_{E_{\tau,\hbar}\times E_{\tau}^{k},a}\\
	&=H^*(\mathfrak{M}^+_v(\alpha,\beta)^{\bm{\rK}_a})\otimes H^*_{\bm{\rK}}(\pt) \otimes _{H^*_{\bm{\rK}}(\pt)} \cO_{E_{\tau,\hbar}\times E_{\tau}^{k},a}\\
	&=H^*(\mathfrak{M}^+_v(\alpha,\beta)^{\bm{\rK}_a})\otimes_{\C} \cO_{E_{\tau,\hbar}\times E_{\tau}^{k},a}
	\end{align*}
	by Lemma \ref{lem:formal_bk}.  We denote by $f^*\mathscr{H}_{\bm{\rK}}$ the pullback of $\mathscr{H}_{\bm{\rK}}$ via the inclusion $f:E_{\tau}^k   \to  E_{\tau,\hbar} \times E_{\tau}^{k}$ to the second component. Then, by Lemma \ref{lem:formal_k} and the fact that $\rK^{\vee}_a=\bm{\rK}_a$ as a subgroup of $\bm{\rK}$, we get
	\begin{align*} 
	(f^*\mathscr{H}_{\bm{\rK}})_a&= H^*(\mathfrak{M}^+_v(\alpha,\beta)^{\bm{\rK}_a})\otimes_{\C} f^{-1}\cO_{E_{\tau,\hbar}\times E_{\tau}^{k},a} \otimes_{f^{-1}\cO_{E_{\tau,\hbar}\times E_{\tau}^{k},a}} \cO_{E_{\tau}^{k},a}\\
	&=H^*(\mathfrak{M}^+_v(\alpha,\beta)^{\bm{\rK}_a}) \otimes _{\C} \cO_{E_{\tau}^{k},a}\\
	&=H^*(\mathfrak{M}^+_v(\alpha,\beta)^{\rK^{\vee}_a}) \otimes _{\C} \cO_{E_{\tau}^{k},a}\\
	&=H^*_{\rK^{\vee}}(\mathfrak{M}^+_v(\alpha,\beta)^{\rK^{\vee}_a})\otimes_{H^*_{\rK^{\vee}}(\pt)} \cO_{E_{\tau}^{k},a}\\
	&=(\mathscr{H}_{\rK^{\vee}})_a.
	\end{align*}
	Since relative spectrum respects pullback, it follows that
	\begin{align*} 
	\mathrm{Ell}_{\rK^{\vee}}(\mathfrak{M}^+_v(\alpha,\beta)) &=\Spec_{E^{\tau}_k} \mathscr{H}_{\rK^{\vee}} 
	=  \Spec_{E^{\tau}_k} f^*\mathscr{H}_{\bm{\rK}} \\
	&= \left (\Spec_{E_{\tau,\hbar}\times E_{\tau}^{k}} \mathscr{H}_{\bm{\rK}} \right) \times_{E_{\tau,\hbar} \times E_{\tau}^{k}}  E_{\tau}^{k} \\
	&=  \mathrm{Ell}_{\bm{\rK}}(\mathfrak{M}^+_v(\alpha,\beta))\times_{E_{\tau,\hbar} \times E_{\tau}^{k}}     E_{\tau}^k. 
	\end{align*}
\end{proof}

We have a similar description for 	$\mathrm{Ell}_{\C^{\times}_{\hbar}}(\mathfrak{M}^+_v(\alpha,\beta))$.

\begin{theorem} \label{thm:ell_cstar_X}
	$\mathrm{Ell}_{\C^{\times}_{\hbar}}(\mathfrak{M}^+_v(\alpha,\beta))$ is the fiber product 
	%$\mathrm{Ell}_{\rA^{\vee}}(\mathfrak{M}^+_v(\alpha,\beta))$ is a subvariety of $\mathrm{Ell}_{\rA^{\vee}\times \C^{\times}_{\hbar}}(\mathfrak{M}^+_v(\alpha,\beta))$ defined by $\vartheta_{0}=0$. In other words, we have
	%\begin{equation} 
	%\mathrm{Ell}_{\rK^{\vee}}(\mathfrak{M}^+_v(\alpha,\beta)) = \Proj \left( \mathrm{R}(E_{\tau}^n)/\left(\prod_{i\in S} \vartheta(x_i) | S\in \mathrm{Circ}(v) \right) \right).
	%\end{equation}
	%Here we have used ``Proj'' to denote the multi-Proj of a multi-graded ring.	
	\begin{equation*} 
	\begin{tikzcd}
	\mathrm{Ell}_{\bm{\rK}}(\mathfrak{M}^+_v(\alpha,\beta))\times_{E_{\tau,\hbar} \times E_{\tau}^{k}}     E_{\tau} \arrow{d}   \arrow{r}  &  \mathrm{Ell}_{\bm{\rK}}(\mathfrak{M}^+_v(\alpha,\beta))
	\arrow{d}  \\
	E_{\tau}   \arrow{r}  &  E_{\tau,\hbar}\times E_{\tau}^{k}
	\end{tikzcd}
	\end{equation*}
	in the category of schemes, where $E_{\tau}\to E_{\tau,\hbar}\times E_{\tau}^{k}$ is the inclusion to the first component and $\mathrm{Ell}_{\bm{\rK}}(\mathfrak{M}^+_v(\alpha,\beta))\to E_{\tau,\hbar}\times E_{\tau}^{k}$ is the map associated to $\mathfrak{M}^+_v(\alpha,\beta)\to \pt$. %by $\mathrm{Ell}_{\bm{\rK}}$. 
\end{theorem}

By the proof of Theorem \ref{thm:ell_X}, this follows from the following lemma:

\begin{lemma}
	For any subgroup $\mathrm{H}\subset \C^{\times}_{\hbar}$, we have 
	\[
	H^*_{\C^{\times}_{\hbar}}(\mathfrak{M}^+_v(\alpha,\beta)^{\mathrm{H}})=H^*(\mathfrak{M}^+_v(\alpha,\beta)^{\mathrm{H}})\otimes \C[\hbar]. 
	\]
\end{lemma}

%By the proof of Theorem \ref{thm:ell_X}, it suffices to prove the following lemma:

%This follows from the equivariant formality of the $\C^{\times}_{\hbar}$-action

\subsection{The isomorphism $\mathrm{Ell}_{\rK^{\vee}}(\mathfrak{M}^+_v(\alpha,\beta))\cong (X_{\vartheta}^n\twobar \rK)^{\rA}$}
%In this subsection, we give a characterization for the $\rT$-fixed point variety in $X_{\vartheta}^n\twobar \rK$. 

%Let $u=\{u_1,\ldots,u_n\}\subset \frak{a}_{\Z}$ be the collection of vectors defining the subtorus $\rK\subset \rT$ in \eqref{eq:ses3}. Let $v=\{v_1,\ldots,v_n\}\subset \frak{k}_{\Z}^{\vee}$ be given by $v_i=\iota^{\vee}(e_i^{\vee})$. $v$ determines the subtorus $\rA^{\vee}\subset \rT^{\vee}$ in \eqref{eq:ses6}. We will assume $u$ is unimodular. This in turn implies that $v$ is also unimodular.

%We define a \textit{circuit} in $u$ to be a minimal subset $S\subset\{1,\ldots,n\}$ such that the vectors $u_i$ for $i\in S$ are linearly dependent. 
%Although we will not use the terminology in this paper, a \textit{cocircuit} in $u$ is defined to be a circuit in $v$.
%We denote by  $\mathrm{Circ}(u)$ the set of circuits in $u$. 

%Let $\cO(e)_i$ be the line bundle over $E_{\tau}^n$ given by the pullback of $\cO(e)$ over $E_{\tau}$ via the projection $\mathrm{pr}_i: E_{\tau}^n\to E_{\tau}$ to the $i$-th component. Let $\mathrm{R}(E^n_{\tau})$ be the $\Z^n$-graded algebra
%\[
%\mathrm{R}(E^n_{\tau})=\bigoplus_{\ell_1,\ldots,\ell_n \ge 0} H^0(E_{\tau}^n, \cO(e)_1^{\ell_1} \otimes \ldots \otimes \cO(e)_n^{\ell_n}), 
%\]
%and let $\vartheta_i\in \mathrm{R}(E^n_{\tau})$ denote the theta function
%\[
%\vartheta_i:\C^n\to \C, \qquad \vartheta_i(x_1,\ldots,x_n)=\vartheta(x_i).
%\]
%Then,
%We will now identify $\mathrm{Ell}_{\rK^{\vee}}(\mathfrak{M}^+_v(\alpha,\beta))$ with the $\rA$-fixed points in $X_{\vartheta}^n\twobar \rK$. 

We now give a description for the $\rA$-fixed points in $X_{\vartheta}^n\twobar \rK$. By an abuse of notations, we again denote by $\cO(e)_i$ be the line bundle over $E_{\tau}^n$ given by the theta function $\vartheta(x_i)$ for $i=1,\ldots,n$. Let $\mathrm{R}(E^n_{\tau})$ be the $\Z^n$-graded algebra
\[
\mathrm{R}(E^n_{\tau})=\bigoplus_{\ell_1,\ldots,\ell_n \ge 0} H^0(E_{\tau}^n, \cO(e)_1^{\ell_1} \otimes \ldots \otimes \cO(e)_n^{\ell_n}).
\]
Then, 

\begin{theorem} \label{thm:fixed_loci}
	We have 
	\begin{equation} \label{eq:fixed_loci_1}
	(X_{\vartheta}^n\twobar \rK)^{\rA}= \Proj \left( \mathrm{R}(E_{\tau}^n)/\left(\prod_{i\in S} \vartheta(x_i) | S\in \mathrm{Circ}(v) \right) \right).
	\end{equation}
	%Here we have used ``Proj'' to denote the multi-Proj of a multi-graded ring.
\end{theorem}

\begin{proof}
	The proof is similar to that of Theorem B.2 in \cite{Hik17}. 
	Let $\cL_i$ be the line bundle over $X_{\vartheta}^n$ given by the
	pullback of the line bundle $\cO(e)$ over $E_{\tau}$ via the composition
	\[
	X_{\vartheta}^n\overset{\mu_{\C}}{\longrightarrow}E_{\tau}^n\overset{\mathrm{pr}_i}{\longrightarrow} E_{\tau}.
	\]
	Let $\mathrm{R}(X_{\vartheta}^n)$ be the $\Z^n$-graded ring
	\[
	\mathrm{R}(X_{\vartheta}^n)=\bigoplus_{\ell_1,\ldots,\ell_n \ge 0} H^0(X_{\vartheta}^n, \cL_1^{\ell_1} \otimes \ldots \otimes \cL_n^{\ell_n}).
	\]
	$\mathrm{R}(X_{\vartheta}^n)$ is an algebra over $\mathrm{R}(E_{\tau}^n)$ with generators $z_i$, $w_i$ for $i=1,\ldots,n$, and subject to the relations 
	\[
	z_iw_i=\vartheta(x_i). 
	\]
	The grading of the generators are given by $\deg z_i=0$ and $\deg w_i=(\delta_{i1},\ldots,\delta_{in})$. By construction, $X_{\vartheta}^n$ is the relative spectrum over $E_{\tau}^n$ of the sheaf of $\cO_{E_{\tau}^n}$-algebras  $\widetilde{\mathrm{R}(X_{\vartheta}^n)}$ associated to $\mathrm{R}(X_{\vartheta}^n)$. 
	
	For the $\rK$-action on $\mathrm{R}(X_{\vartheta}^n)$, the $\rK$-invariant subring $\mathrm{R}(X_{\vartheta}^n)^{\rK}$ is an algebra over $\mathrm{R}(E_{\tau}^n)$ with generators the $\rK$-invariant monomials, which are of the form
	\[
	r^{\lambda}:=\prod_i^n \left(z_i^{\max(\lambda_i,0)}  w_i^{\max(-\lambda_i,0)} \right), \qquad \lambda \in \frak{a}^{\vee}_{\Z}\subset \frak{t}^{\vee}_{\Z} ,  
	\] 
	and subject to the relations
	\begin{equation*} 
	r^{\lambda}r^{\mu}=r^{\lambda+\mu}\prod_i^n \vartheta_i^{\delta(\lambda_i,\mu_i)}, 
	\end{equation*} 
	where $\delta(\lambda_i,\mu_i)$ is defined by \eqref{eq:delta}. The relative spectrum of $\widetilde{\mathrm{R}(X_{\vartheta}^n)^{\rK}}$ over $E_{\tau}^n$ gives $X_{\vartheta}^n\twobar \rK$.
	
	Now, let's consider the $\rA$-action on $\mathrm{R}(X_{\vartheta}^n)^{\rK}$ induced by the residual $\rA$-action on $X_{\vartheta}^n\twobar \rK$. The $\rA$-coinvariant ring $(\mathrm{R}(X_{\vartheta}^n)^{\rK})_{\rA}$ is defined by
	\[
	(\mathrm{R}(X_{\vartheta}^n)^{\rK})_{\rA}=\mathrm{R}(X_{\vartheta}^n)^{\rK}/\left(  r^{\lambda}-t\cdot r^{\lambda}| t\in \rA, \lambda \in \frak{a}_{\Z}^{\vee}  \right).
	\]
	It is easy to see that all the generators $r^{\lambda}$ are contained in the ideal. Thus, we have an ring isomorphism
	\begin{equation} \label{eq:fixed_loci_2}
	(\mathrm{R}(X_{\vartheta}^n)^{\rK})_{\rA}=\mathrm{R}(E_{\tau}^n)/\left( \prod_i^n \vartheta(x_i)^{\delta(\lambda_i,-\lambda_i)}| \lambda\in \frak{a}_{\Z}^{\vee}  \right).
	\end{equation}
	Since we have
	\begin{equation} 
	(X_{\vartheta}^n\twobar \rK)^{\rA} =\Proj(\mathrm{R}(X_{\vartheta}^n)^{\rK})_{\rA},
	\end{equation}
	it remains to show that $(\mathrm{R}(X_{\vartheta}^n)^{\rK})_{\rA}$ is isomorphic to the ring in the RHS of \eqref{eq:fixed_loci_1}. 
	
	Let $S\in \mathrm{Circ}(v)$ be a circuit. There is a relation $\sum_{i\in S} a_i v_i =0$, where $a_i\in \Z$ is nonzero. Since $v$ is unimodular, we can assume $a_i=\pm 1$. We set $a_i=0$ for $i\notin S$ and let $\vec{a}\in \frak{t}_{\Z}$ be the vector $\vec{a}=(a_1,\ldots,a_n)$. Since $\vec{a}\in \ker(\frak{t}_{\Z}^{\vee}\to \frak{k}_{\Z}^{\vee})=\Im(\frak{a}_{\Z}^{\vee}\to\frak{t}_{\Z}^{\vee})$, we have $r^{\vec{a}}r^{-\vec{a}}=\prod_{i\in S} \vartheta(x_i)$ in $\mathrm{R}(X_{\vartheta}^n)^{\rK}$ and $\prod_{i\in S} \vartheta(x_i)=0$ in $(\mathrm{R}(X_{\vartheta}^n)^{\rK})_{\rA}$. Conversely, if a monomial $\prod_{i=1}^n \vartheta(x_i)^{b_i}$ vanishes in $(\mathrm{R}(X_{\vartheta}^n)^{\rK})_{\rA}$, then we must have 
	\[
	\prod_{i=1}^n \vartheta(x_i)^{b_i}=r^{\lambda} r^{-\lambda} \prod_{i=1}^n \vartheta(x_i)^{c_i},
	\]
	for some $(c_1,\ldots,c_n)\in \Z^n$ and $\lambda \in \frak{a}_{\Z}^{\vee}$ with $\lambda \ne 0$. 
	This means the set $\{i \mid c_i\ne 0\}$ is contained in $\{i \mid \lambda_i\ne 0\}$. Since $\sum_{i=1}^n \lambda_i v_i=0$, the latter set is generated by circuits in $\mathrm{Circ}(v)$. Hence $\prod_{i=1}^n \vartheta(x_i)^{b_i}$ is contained in the ideal generated by $\prod_{i\in S}\vartheta(x_i)$ for $S\in \mathrm{Circ}(v)$. It follows that 
	\[
	\mathrm{R}(E_{\tau}^n)/\left(\prod_{i\in S} \vartheta(x_i) | S\in \mathrm{Circ}(v) \right) \cong \mathrm{R}(E_{\tau}^n)/\left( \prod_i^n \vartheta(x_i)^{\delta(\lambda_i,-\lambda_i)}| \lambda\in \frak{a}_{\Z}^{\vee}  \right).
	\]
\end{proof}

It follows from \ref{thm:ell_X} and Theorem \ref{thm:fixed_loci} that

\begin{corollary} \label{cor:Hikita}
	There is an isomorphism
	\begin{equation*}
	\mathrm{Ell}_{\rK^{\vee}}(\mathfrak{M}^+_v(\alpha,\beta))\cong (X_{\vartheta}^n\twobar \rK)^{\rA}.
	\end{equation*}
\end{corollary}

\begin{remark}
	Hikita's conjecture was extended by Nakajima to an isomorphism between $H^*_{\C^{\times}_{\hbar}}(X^!)$ and the $B$-algebra of the quantized coordinate ring of $X$ (see \cite{KTWWY19}).  %for the conical $\C^{\times}$-action 
	Therefore, we expect $\mathrm{Ell}_{\bm{\rK}}(\mathfrak{M}^+_v(\alpha,\beta))$ and $\mathrm{Ell}_{C^{\times}_{\hbar}}(\mathfrak{M}^+_v(\alpha,\beta))$ to correspond to the quantization of $X_{\vartheta}^n\twobar \rK$ and $\mathfrak{M}^{\tau}_u(0,0)$, respectively. This is an interesting problem in its own right and we hope to study it in a future work. %respectively, which in turn are related to theory of elliptic $r$-matrices (e.g. \cite{Jor14}). 
\end{remark}

\appendix
\section{Additive and Multiplicative hypertoric varieties} \label{sec:appendix}
We give a brief review of additive and multiplicative hypertoric varieties in this appendix. We will assume the same combinatorial input as given in the beginning of subsection \ref{sec:ell_hypertoric}.
\subsection{Additive hypertoric varieties}
Let's consider the $T^*\C^n$ equipped with the standard hyperkähler structure. We have a Kähler form
\begin{equation*} %\label{eq_kahler_form_+}
\omega_{\R}^+=\frac{\sqrt{-1}}{2}\sum_{i=1}^n (dz_i\wedge d\bar{z}_i+dw_i\wedge d\bar{w}_i),
\end{equation*}
and holomorphic symplectic form 
\[
\omega_{\C}^+=\sum_{i=1}^n dz_i\wedge dw_i.
\]
Here $z_i$ and $w_i$ are the base and fiber coordinates, respectively. The action of $\rT$ (resp. the compact real torus $\rT_{\R}$ in $\rT$) on $T^*\C^n$ defined by
\begin{equation} \label{eq:T_action}
\vec{t}\cdot(\vec{z},\vec{w})=(t_1z_1,t_1^{-1}w_1,\ldots,t_nz_n,t_n^{-1}w_n), %\qquad t=(t_1,\ldots,t_n)\in T,
\end{equation}
is Hamiltonian with respect to $\omega_{\C}^+$ (resp. $\omega_{\R}^+$). The hyperkähler moment map 
\[
(\mu_{\R,\rK}^+,\mu_{\C,\rK}^+):T^*\C^n \to \mathfrak{k}_{\R}^{\vee}\oplus\mathfrak{k}^{\vee}
\]
for the action of the subtorus $K\subset T^n$ is given by
\[
\mu_{\R,\rK}^+ (\vec{z},\vec{w})=\frac{1}{2}\sum_{i=1}^n(|z_i|^2-|w_i|^2)\iota^{\vee}e_i^{\vee},
\] 
\[
\mu_{\C,\rK}^+ (\vec{z},\vec{w})=\sum_{i=1}^n (z_iw_i)\iota^{\vee}e_i^{\vee}.
\]

\begin{defn}[Additive hypertoric varieties]\label{defn:+_hypertoric}
	Given a collection $u=\{u_1,\ldots,u_n\}$ of primitive vectors in $\mathfrak{a}_{\Z}$, and a choice of parameters $(\alpha,\beta)\in \mathfrak{k}_{\R}^{\vee}\oplus \mathfrak{k}^{\vee}$,
	we defined the associated (additive) hypertoric variety $\mathfrak{M}^+_u(\alpha,\beta)$ to be the (possibly singular) hyperkähler quotient 
	\[
	\mathfrak{M}_u^+(\alpha,\beta)= T^*\C^n\threebar_{(\alpha,\beta)}\rK :=(\mu_{\R,\rK}^+,\mu_{\C,\rK}^+)^{-1}(\alpha,\beta)/\rK_{\R}
	\]
\end{defn}

Equivalently, as a holomorphic symplectyc variety, we can define $\mathfrak{M}_u^+(\alpha,\beta)$ to as the GIT quotient
\[
\mathfrak{M}_u^+(\alpha,\beta)=(\mu_{\C,\rK}^+)^{-1}(\beta)\twobar_{\alpha} \rK. %=\mathrm{Proj} \left(\bigoplus_{k=0}^{\infty}\C[\mu_{\C}^{-1}(\lambda_{\C})]^{\lambda_{\R}^k}\right),
\]

%The quotient torus $T^d=T^n/K$ acts on $\mathfrak{M}_{u,\lambda}$ with the hyper-Kähler moment map 
%\[
%(\bar{\mu}_{\R},\bar{\mu}_{\C}):\mathfrak{M}_{u,\lambda}\to (\mathfrak{t}^d)^*\oplus (\mathfrak{t}^d_{\C})^*
%\]
%given by 
%\[
%(\bar{\mu}_{\R},\bar{\mu}_{\C})[z,w]=\frac{1}{2}\sum_{i=1}^n(|z_i|^2-|w_i|^2+\hat{\lambda}_{\R,i})\check{e}_i\oplus\sum_{i=1}^n(z_iw_i+\hat{\lambda}_{\C,i})\check{e}_i\in\Ker \iota^* \oplus\Ker (\iota^*_{\C}=(\mathfrak{t}^d)^*\oplus (\mathfrak{t}^d_{\C})^*,
%\]
%where $((\hat{\lambda}_{\R,1},\ldots,\hat{\lambda}_{\R,n}),(\hat{\lambda}_{\C,1},\ldots,\hat{\lambda}_{\C,n}))\in(\mathfrak{t}^n)^*\oplus (\mathfrak{t}^n_{\C})^*$ is a lift of $\lambda$. Note that this map is always surjective.

\subsection{Multiplicative hypertoric varierties} \label{sec_2.2}
For multiplicative hypertoric varieties, %instead of $T^*\C^n$, 
we consider the space
\[
(T^*\mathbb{C}^{n})^{\circ}=\{(z,w)\in T^*\mathbb{C}^n|1-z_iw_i \ne 0, i=1,\ldots,n\}.
\]
equipped with the holomorphic symplectic form
\[
\omega_{\C}^{\times}=\sum_i^n \cfrac{dz_i\wedge dw_i}{1-z_iw_i},
\] 
and the Kähler form $\omega_{\R}^{\times}$ which is the restriction of $\omega_{\R}^{+}$  to $(T^*\mathbb{C}^{n})^{\circ}$.

$(T^*\mathbb{C}^{n})^{\circ}$ is a $\rT$-invariant open subset of $\rT$-action on $T^*\C^n$ defined in \eqref{eq:T_action}. Moreoever, the $\rT$-action on $(T^*\mathbb{C}^{n})^{\circ}$ is quasi-Hamiltonian with respect to $\omega_{\C}^{\times}$, and there is a torus-valued moment map $\mu_{\C}^{\times}:(T^*\mathbb{C}^{n})^{\circ}\to \rT$ 
\[
\mu_{\C}^{\times}(\vec{z},\vec{w})= (1-z_1w_1,\ldots,1-z_nw_n). 	
\] 
The inclusion map $\phi:\rK\to \rT$ for the subtorus $\rK$ has the form \eqref{eq:K}. We define its transpose $\phi^{\vee}_{\tau}: \rT\to \rK$ by
\[
\phi^{\vee}_{\tau}(t_1,\ldots,t_n) = \left(\prod_{i=1}^n t_i^{\phi_{i1}},\ldots,\prod_{i=1}^n t_i^{\phi_{ik}}\right).
\]
The $\rK$-action on $(T^*\mathbb{C}^{n})^{\circ}$ 
is quasi-Hamiltonian with the torus-valued moment map $\mu_{\C,\rK}^{\times}$ given by $\mu_{\C,\rK}^{\times}=\phi^{\vee}_{\tau}\circ \mu_{\C}^{\times}$.

\begin{defn}[Multiplicative hypertoric varieties] \label{defn:*_hypertoric}
	Given a collection $u=\{u_1,\ldots,u_n\}$ of primitive vectors in $\mathfrak{a}_{\Z}$, and a choice of parameters $(\alpha,\beta)\in \mathfrak{k}_{\R}^{\vee}\times \rK$,
	we defined the associated multiplicative hypertoric variety $\mathfrak{M}_u^{\times}(\alpha,\beta)$ either by the symplectic quotient
	\[
	\mathfrak{M}_u^{\times}(\alpha,\beta):=(\mu_{\R,\rK}^{\times},\mu_{\C,\rK}^{\times})(\alpha,\beta)/\rK_{\R},
	\]
	or the GIT quotient 
	\[
	\mathfrak{M}_u^{\times}(\alpha,\beta):= (\mu_{\C,\rK}^{\times})^{-1} (\beta)\twobar_{\alpha} \rK.
	\]
\end{defn}

\bibliographystyle{amsalpha}
\bibliography{geometry}	
\end{document}